\documentclass[11pt]{article}
\usepackage{pb-diagram}
\usepackage{amsmath,amsthm,amsfonts,amssymb,latexsym}
\usepackage{amsmath}
\usepackage{hyperref}
\usepackage{a4wide}
\usepackage[all]{xy}

\newtheorem{thm}{Theorem}[section]
\newtheorem{cor}[thm]{Corollary}
\newtheorem{lem}[thm]{Lemma}
\newtheorem{prop}[thm]{Proposition}
\newtheorem{ex}[thm]{Example}
\newtheorem{defn}[thm]{Definition}
\newtheorem{defns}[thm]{Definitions}
\newtheorem{rmk}[thm]{Remark}

\newtheorem{theorem}{Theorem}[section]
\newtheorem{definition}[theorem]{Definition}

\newtheorem{remark}[theorem]{Remark}

\newcommand{\ad}{\mathrm{ad}}

\begin{document}

\title{Representations, extensions and deformations of $n$-BiHom-Lie algebras}
\author{Ismail Laraiedh \thanks{Departement of Mathematics, Faculty of Sciences, Sfax University, BP 1171, 3000 Sfax, Tunisia. E.mail:
Ismail.laraiedh@gmail.com and Departement of Mathematics, College of Sciences and Humanities - Kowaiyia, Shaqra University,
Kingdom of Saudi Arabia. E.mail:
ismail.laraiedh@su.edu.sa}}
\maketitle
\begin{abstract}
In this paper we define and discuss the representations of $n$-BiHom-Lie algebra. We also introduce $T_{\theta}$-extensions and $T_{\theta}^{\ast}$-extensions of $n$-BiHom-Lie algebras and prove the necessary and sufficient conditions for
a $2m$-dimensional quadratic $n$-Bihom-Lie algebra to be isomorphic to a $T_{\theta}^{\ast}$-extension. Moreover, we develop the one-parameter formal deformations of $n$-BiHom-Lie algebras, and we proved that the first
and second cohomology groups are suitable to the deformation theory involving
infinitesimals, equivalent deformations, and rigidity



\end{abstract}

\section{Introduction}

The n-Lie algebras found their applications in many fields of Mathematics and Physics. Ternary Lie algebras appeared first in Nambu generalization of Hamiltonian mechanics \cite{Nambu:GenHD}
using ternary bracket generalization of Poisson algebras.  Nambu mechanics \cite{Nambu:GenHD} involves an $n$-ary product that satisfies the $n$-ary Nambu identity, which is an $n$-ary generalization of the Jacobi identity.  Bagger-Lambert algebras \cite{bl} are ternary Nambu algebras with some extra structures, and they appear in the study of string theory and $M$-branes.
The cohomology of $n$-Lie algebras, generalizing the Chevalley-Eilenberg Lie algebras cohomology, was  introduced by Takhtajan \cite{Tcohomology} in its simplest form, later a complex adapted
to the study of formal deformations was introduced by Gautheron \cite{Gautheron}, then reformulated
by Daletskii and Takhtajan \cite{DT} using the notion of base Leibniz algebra of an $n$-Lie algebra.

Hom-type generalizations of $n$-ary Nambu(-Lie) algebras, called $n$-ary Hom-Nambu(-Lie) algebras, were introduced by Ataguema, Makhlouf, and Silvestrov in \cite{ams} .  Each $n$-ary Hom-Nambu(-Lie) algebra has $n-1$ linear twisting maps, which appear in a twisted generalization of the $n$-ary Nambu identity called the $n$-ary Hom-Nambu identity. If the twisting maps are all equal to the identity, one recovers an $n$-ary Nambu(-Lie) algebra.

A generalization of this approach led the authors of \cite{GRAZIANI} to introduce BiHom-algebras, which are algebras where the identities defining the structure are twisted by two homomorphisms $\alpha$ and $\beta$. These
algebraic structures include BiHom-associative algebras, BiHom-Lie algebras and BiHom-bialgebras. More applications of BiHom-algebras, BiHom-Lie superalgebras and BiHom-Novikov algebras can be
found in \cite{S S, Guo1, chtioui1,chtioui2, ismail, Li C, Liu2}. BiHom-type generalizations of n-ary Nambu-Lie algebras, called n-ary BiHom-Nambu-
Lie algebras, were introduced by Kitouni, Makhlouf, and Silvestrov in \cite{KMSi}. Each n-ary
BiHom-Nambu-Lie algebra has (n - 1)-linear twisting maps, which appear in a twisted
generalization of the n-ary Nambu identity called the n-ary BiHom-Nambu identity.

The representation theory of an algebraic object is very important since it reveals some of its
profound structures hidden underneath.
The notion of a representation of a $3$-BiHom-Lie algebra was introduced in \cite{Li C}.
In
this paper, we define and discusses representations of $n$-BiHom-Lie algebras
Also
The extension is an important way to find a larger algebra and there are many extensions such as double extensions and Kac-Moody extensions, etc. In 1997, Bordemann introduced the notion of $T^*$-extensions of Lie algebras\cite{B} and proved that every nilpotent finite-dimensional algebra over an algebraically closed field carrying a nondegenerate invariant symmetric bilinear form is a suitable $T^*$-extension. The method of $T^*$-extension was used in \cite{BBM} and was generalized to many other algebras recently\cite{lyc,YL}. Furthermore, deformation problems appear in various areas of mathematics, especially in algebra, algebraic and analytic geometry, and mathematical physics. The deformation theory was introduced by Kodaira and Spencer to study complex structures of higher dimensional manifolds (see \cite{Kodaira&Spencer}), which was extended to rings and algebras by Gerstenhaber in \cite{Gerstenhaber1, Gerstenhaber4} and to Lie algebras by Nijenhuis and Richardson in \cite{Nijenhuis&Richardson}. They connected deformation
theory for associative algebras and Lie algebras with Hochschild cohomology and Chevally-Eilenberg cohomology, respectively. See also \cite{Ammar&Ejbehi&Makhlouf, Elhamdadi&Makhlouf, Kubo&Taniguchi, Ma&Chen&Lin, Makhlouf&Silvestrov, deforr} for more deformation theory.

The paper is organized as follows. In Section 2, summarizes basic concepts and recall the definition
of $n$-Bihom-Lie algebras. In Section 3,
we give the definition of representations of $n$-Bihom-Lie algebras and we can obtain the semidirect product $n$-Bihom-Lie algebra associated with any representation $\rho$ of an $n$-Bihom-Lie algebra $\mathfrak{g}$ on a vector space $V$. In Section 4, $T_\theta^*$-extensions and $T_\theta$-extensions of $n$-Bihom-Lie algebras are studied. We give some properties and we prove the necessary and sufficient
conditions for a $2m$-dimensional quadratic $n$-Bihom-Lie algebra to be isomorphic to a $T_\theta^*$-extension. In section 5, we define low orders coboundary operators and cohomology groups of
$n$-BiHom-Lie algebras and develop the $1$-parameter formal deformation theory.

In this paper, we work over an algebraically closed field $\mathbb{K}$ of characteristic 0 and all the vector spaces are over $\mathbb{K}$. Everywhere hereafter, the notation $\widehat{x}$ means that $x$ is excluded, for example, we write $f(x_{1},\cdots,\widehat{x_{i}},\cdots,x_{n})$ for $f(x_{1},\cdots,x_{i-1},x_{i+1},\cdots,x_{n})$.

\section{Preliminaries}
This section contains necessary important basic notions and notations which will be used in next sections.
\begin{definition}\cite{KMSi}
An $n$-BiHom-Lie algebra is a vector space $\mathfrak{g}$, equipped with an $n$-linear operation $[\cdot,...,\cdot]$ and two linear maps $\alpha$ and $\beta$ satisfying the following conditions:
\begin{enumerate}
\item $\alpha \circ \beta = \beta \circ \alpha$.
\item $\forall x_1,...,x_n \in \mathfrak{g}, \alpha([x_1,...,x_n]_{\mathfrak{g}}) = [\alpha(x_1),...,\alpha(x_n)]_{\mathfrak{g}}$ and $\beta([x_1,...,x_n]_{\mathfrak{g}}) = [\beta(x_1),...,\beta(x_n)]_{\mathfrak{g}}$.
\item BiHom-skewsymmetry: $\forall x_1,...,x_n \in \mathfrak{g}$,

$\begin{array}{lllll}
&&[\beta(x_{1}),\ldots,\beta(x_{k}),\beta(x_{k+1}),\ldots,\beta(x_{n-1}),\alpha(x_{n})]_{\mathfrak{g}}\\
&=&-[\beta(x_{1}),\ldots,\beta(x_{k+1}),\beta(x_{k}),\ldots,\beta(x_{n-1}),\alpha(x_{n})]_{\mathfrak{g}}\\
&=&-[\beta(x_{1}),\ldots,\beta(x_{n-2}),\beta(x_{n}),\alpha(x_{n-1})]_{\mathfrak{g}},\end{array}$

where $k=1,2, \dots , n-2.$
\item $n$-BiHom-Jacobi identity: $\forall x_1,..,x_{n-1},y_1,...,y_n \in \mathfrak{g}$,
\begin{align*}
&[\beta^2(x_1),...,\beta^2(x_{n-1}),[\beta(y_1),...,\beta(y_{n-1}),\alpha(y_n)]_{\mathfrak{g}}]_{\mathfrak{g}}\\
& =\sum_{k=1}^n (-1)^{n-k} [\beta^2(y_1),...,\widehat{\beta^2(y_k)},...,\beta^2(y_n),[\beta(x_1),...,\beta(x_{n-1}),\alpha(y_k)]_{\mathfrak{g}}]_{\mathfrak{g}}.
\end{align*}
\end{enumerate}
\end{definition}

\begin{remark}
\begin{enumerate}
\item
When $\alpha=\beta=id$, we get $n$-Lie algebra
\item
For $n=2$, the BiHom-Jacobi identity is:
\begin{equation}\label{ll}[\beta^{2}(x),[\beta(y_{1}),\alpha(y_{2}]_{\mathfrak{g}}]_{\mathfrak{g}}=-[\beta^{2}(y_{2}),[\beta(x),\alpha(y_{1}]_{\mathfrak{g}}]_{\mathfrak{g}}
+[\beta^{2}(y_{1}),[\beta(x),\alpha(y_{2}]_{\mathfrak{g}}]_{\mathfrak{g}},\end{equation}
and the identity (\ref{ll}) is equivalent to
$$\displaystyle{\circlearrowleft_{x,y_{1},y_{2}}}[\beta^{2}(x),[\beta(y_{1}),\alpha(y_{2}]_{\mathfrak{g}}]_{\mathfrak{g}}=0.$$
\item
For $n=3$, the $3$-BiHom-Jacobi identity is:
$$
\begin{array}{lllllll}&&[\beta^{2}(x_{1}),\beta^{2}(x_{2}),[\beta(y_{1}),\beta(y_{2}),\alpha(y_{3})]_{\mathfrak{g}}]_{\mathfrak{g}}\\
&=&[\beta^{2}(y_{2}),\beta^{2}(y_{3}),[\beta(x_{1}),\beta(x_{2}),\alpha(y_{1})]_{\mathfrak{g}}]_{\mathfrak{g}}\\
&&-[\beta^{2}(y_{1}),\beta^{2}(y_{3}),[\beta(x_{1}),\beta(x_{2}),\alpha(y_{2})]_{\mathfrak{g}}]_{\mathfrak{g}}\\
&&+[\beta^{2}(y_{1}),\beta^{2}(y_{2}),[\beta(x_{1}),\beta(x_{2}),\alpha(y_{3})]_{\mathfrak{g}}]_{\mathfrak{g}},\end{array}
$$
\end{enumerate}
\end{remark}
\begin{theorem}\label{YauTwist}\cite{KMSi}
Let $(\mathfrak{g},[\cdot,...,\cdot]_{\mathfrak{g}})$ be an $n$-Lie algebra, and let $\alpha,\beta : \mathfrak{g} \to \mathfrak{g}$ be algebra morphisms such that $\alpha \circ \beta = \beta \circ \alpha$. The algebra $(\mathfrak{g}, [\cdot,...,\cdot]_{\alpha\beta},\alpha,\beta)$, where $[\cdot,...,\cdot]_{\alpha\beta}$ is defined by
\[[x_1,...,x_n]_{\alpha\beta} = [\alpha(x_1),...,\alpha(x_{n-1}),\beta(x_n)]_{\mathfrak{g}},\]
is an $n$-BiHom-Lie algebra.
\end{theorem}
\begin{ex}
We consider the $5$-dimensional $3$-Lie algebra defined with respect to a basis $\mathfrak{g}=(e_1,e_2,e_3,e_4,e_5)$, by:
\[ [e_2,e_3,e_4]_\mathfrak{g}=e_1 ; [e_2,e_4,e_5]_\mathfrak{g}=-e_2 ; [e_3,e_4,e_5]_\mathfrak{g}=e_3.\]

We have two morphisms $\alpha$, $\beta$ of this algebra defined, with respect to the same basis, by:
\[ [\alpha]=\begin{pmatrix}-1 &0 & 0 & 0 & 0 \\ 0 & 1&0  & 0 & 0 \\ 0 & 0  & 1&0 & 0 \\ 0 & 0  & 0& 1&0  \\ 0 & 0 & 0 & 0 &-1 \end{pmatrix} ;
[\beta]= \begin{pmatrix} 0 & 0 & 0 & 0 & -1 \\ 0 & 0 & 0 & -1 & 0 \\ 0 & 0 & -1 & 0 & 0 \\  0 &1 & 0 & 0 & 0 \\  1 &0 & 0 & 0 & 0 \end{pmatrix}. \]
One can easily check that $\alpha$ and $\beta$ commute, then $(\mathfrak{g}, [\cdot,\cdot,\cdot]_{\alpha\beta},\alpha,\beta)$ is a $3$-BiHom-Lie algebra.
\end{ex}
\begin{definition}
 A morphism   $f : (\mathfrak{g}, [\cdot, \cdots, \cdot]_{\mathfrak{g}}, \alpha,\beta)\rightarrow (\mathfrak{g}', [\cdot, \dots, \cdot]_{\mathfrak{g}'}, \alpha',\beta')$ of an $n$-BiHom-Lie algebras
 is a linear map $f :\mathfrak{g}\rightarrow \mathfrak{g}'$ such that
$f\circ\alpha=\alpha'\circ f,~~f\circ\beta=\beta'\circ f$  and for any $x_k\in\mathfrak{g}$,
$$f([x_1, \cdots, x_n]_{\mathfrak{g}})=[f(x_1), \cdots, f(x_n)]_{\mathfrak{g}'}$$
\end{definition}
\begin{defns}
Let $(\mathfrak{g},[\cdot,\cdots,\cdot]_{\mathfrak{g}},\alpha,\beta)$ an $n$-BiHom-Lie algebra.
\begin{enumerate}
\item
A subspace $H$ of $\mathfrak{g}$ is an $n$-BiHom-subalgebra of $(\mathfrak{g},[.,\cdots,.]_{\mathfrak{g}},\alpha,\beta)$ if $\alpha(H)\subset H$,
$\beta(H)\subset H$ and $[H,\cdots,H]_{\mathfrak{g}}\subseteq H$, (i.e.,
$[x_{1},\cdots,x_{n}]_{\mathfrak{g}}\in H,~~\forall x_{k}\in H$).
\item
A subspace $I\subset \mathfrak{g}$ is an $n$-BiHom ideal of $(\mathfrak{g},[.,\cdots,.]_{\mathfrak{g}},\alpha,\beta)$ if $\alpha(I)\subset I$, $\beta(I)\subset I$\\
and $[I,\mathfrak{g},\cdots,\mathfrak{g}]_{\mathfrak{g}}\subseteq I$, (i.e.
$[x,y_{1},\cdots,y_{n-1}]_{\mathfrak{g}}\in I,~~\forall x\in I;~y_{k}\in \mathfrak{g}$.
\item
The center of $(\mathfrak{g},[.,\cdots,.]_{\mathfrak{g}},\alpha,\beta)$ is the set of $x\in \mathfrak{g}$ such that\\ $[x,y_{1},\cdots,y_{n-1}]_{\mathfrak{g}}=0,$ for any $y_{i}\in \mathfrak{g}$. The center is an ideal of $\mathfrak{g}$ which we will denote by $Z(\mathfrak{g})$.
\end{enumerate}
\end{defns}
\begin{defns}\cite{KMSi}
Let $(\mathfrak{g},[\cdot,\cdots,\cdot]_{\mathfrak{g}},\alpha,\beta)$ be an $n$-BiHom-Lie algebra.
We define, for all $X = x_1\wedge ... \wedge x_{n-1}$, $Y = y_1\wedge ... \wedge y_{n-1}$ in $\wedge^{n-1}\mathfrak{g}$ and $z\in \mathfrak{g}$ the following:
\begin{enumerate}
\item The action of fundamental objects on $\mathfrak{g}$:
\[ X \cdot z = ad_X(z) = [x_1,...,x_{n-1},z]_{\mathfrak{g}}. \]
\item The linear maps $\bar{\alpha},\bar{\beta} : \wedge^{n-1}\mathfrak{g} \to \wedge^{n-1}\mathfrak{g}$:
\[\widetilde{\alpha}(X) = \alpha(x_1)\wedge...\wedge \alpha(x_{n-1}) \text{ and } \widetilde{\beta}(X) = \beta(x_1)\wedge...\wedge \beta(x_{n-1}).\]
\end{enumerate}
\end{defns}

\section{Representations of $n$-BiHom-Lie algebras}
\begin{defn}\label{defi:rep}
A {\bf representation} of an $n$-BiHom-Lie algebra $(\mathfrak{g},[\cdot,\cdots,\cdot]_\mathfrak{g},\alpha,\beta)$ on a vector space $V$ with respect to a endomorphisms $\alpha_{V},\beta_{V}\in End(V)$ is a linear map $\rho:\wedge ^{n-1}\mathfrak{g}\longrightarrow End(V)$ such that for all $X=x_1\wedge\cdots\wedge x_{n-1},~Y=y_1\wedge\cdots\wedge y_{n-1}\in\wedge^{n-1}\mathfrak{g}$, we have
\begin{enumerate}
  \item$\rho(\widetilde{\alpha}(X) )\circ\alpha_{V}=\alpha_{V}\circ \rho(X);$
 \item$\rho(\widetilde{\beta}(X) )\circ\beta_{V}=\beta_{V}\circ \rho(X);$

  \item \begin{eqnarray*} &&\rho(\widetilde{\alpha\beta}(X))\circ\rho(Y)-\rho(\widetilde{\beta}(Y))\circ \rho(\widetilde{\alpha}(X))\\&=&\sum_{i=1}^{n-1}\rho(\beta(y_1),\cdots,\beta(y_{i-1}),[\beta(x_1),\cdots,\beta(x_{n-1}),y_i]_\mathfrak{g},\beta(y_{i+1})
 ,\cdots,\beta(y_{n-1}))\circ \beta_{V};\end{eqnarray*}

    \item\begin{eqnarray*}
        &&\rho(\beta(y_2),\cdots,\beta(y_{n-1}),[\beta(x_1),\cdots,\beta(x_{n-1}),y_1]_\mathfrak{g})\circ\beta_{V}\\
        &=&\sum_{i=1}^{n-1}(-1)^{n-i}\rho(\alpha\beta(x_1),\cdots,\widehat{\alpha\beta}(x_i),\cdots,\alpha\beta(x_{n-1}),\beta(y_1))\circ\rho(y_2,\cdots,y_{n-1},\alpha(x_i))\\
        &+&\rho(\widetilde{\alpha\beta}(X))\circ\rho(Y).
        \end{eqnarray*}
\end{enumerate}
\end{defn}
We denote a representation by $(V,\rho,\alpha_{V},\beta_{V})$.
\begin{prop} Let $(\mathfrak{g},[\cdot,\cdots,\cdot]_\mathfrak{g},\alpha,\beta)$ be an $n$-BiHom-Lie algebra.
Define $ad:\wedge^{n-1}\mathfrak{g}\longrightarrow End(\mathfrak{g})$ by
\begin{equation}
  \ad_X(y)=[x_1,\cdots,x_{n-1},y]_{\mathfrak{g}},\quad\forall X=x_1\wedge\cdots \wedge x_{n-1}\in\wedge^{n-1}\mathfrak{g},~y\in\mathfrak{g}.
\end{equation}
Then $(\mathfrak{g},\ad,\alpha,\beta)$ is a representation of the $n$-BiHom-Lie algebra $(\mathfrak{g},[\cdot,\cdots,\cdot]_\mathfrak{g},\alpha,\beta)$ on $\mathfrak{g}$, called {\bf adjoint representation}.
\end{prop}
\begin{proof}
Follows a direct computation by the definition of representations.
\end{proof}

\begin{prop}\label{prop11}
Let  $(V,\rho,\alpha_{V},\beta_{V})$ be a representation of an $n$-BiHom-Lie algebra $(\mathfrak{g},[\cdot,\cdots,\cdot]_\mathfrak{g},\alpha,\beta)$. Assume that the maps $\alpha$ and $\beta_{V}$ are bijective. Then $(\mathfrak{g}\oplus V,[\cdot,\cdots,\cdot]_{\rho},\alpha+\alpha_{V},\beta+\beta_{V})$ is an $n$-BiHom-Lie algebra, where $\alpha+\alpha_{V},\beta+\beta_{V}: \mathfrak{g}\oplus V\rightarrow\mathfrak{g}\oplus V$ are defined by
$(\alpha+\alpha_{V})(x+u)=\alpha(x)+\alpha_{V}(u)$ and $(\beta+\beta_{V})(x+u)=\beta(x)+\beta_{V}(u)$
and the bracket operation $[\cdot,\cdots,\cdot]_{\rho}:\wedge^n(\mathfrak{g}\oplus V)\longrightarrow\mathfrak{g}\oplus V$ is defined by
$$\begin{array}{llll}[x_1+u_1,\cdots,x_n+u_n]_{\rho}=[x_1,\cdots,x_n]_\mathfrak{g}&+&\displaystyle{\sum_{i=1}^{n-1}}(-1)^{n-i}\rho(x_1,\cdots,\widehat{x}_i,\cdots,x_{n-1},\alpha^{-1}\beta(x_n))
(\alpha_{V}\beta^{-1}_{V}(u_i))\\
&+&\rho(x_1,\cdots,x_{n-1})(u_n)
,\,\,\,\,\forall x_i\in\mathfrak{g},u_i\in V.\end{array}.$$ We denote this semidirect $n$-BiHom-Lie algebra simply by $\mathfrak{g}\ltimes V$.
\end{prop}

\begin{proof}
First we show $(\alpha+\alpha_{V})\circ(\beta+\beta_{V})=(\beta+\beta_{V})\circ(\alpha+\alpha_{V})$ from the fact $\alpha\circ\beta=\beta\circ\alpha,$ and $\alpha_{V}\circ\beta_{V}=\beta_{V}\circ\alpha_{V}$.

Now, we show that $\alpha+\beta$ is an algebra morphism. On one hand, we have
\begin{eqnarray*}
&&(\alpha+\alpha_{V})[x_1+u_1,\cdots,x_n+u_n]_{\rho}\\&=&(\alpha+\alpha_{V})([x_1,\cdots,x_n]_\mathfrak{g}+\sum_{i=1}^{n}(-1)^{n-i}\rho(x_1,\cdots,\widehat{x}_i,\cdots,x_{n-1},
\alpha^{-1}\beta(x_n))(\alpha_{V}\beta^{-1}_{V}(u_i))\\
                                             &=&\alpha([x_1,\cdots,x_n]_\mathfrak{g})+\sum_{i=1}^{n}(-1)^{n-i}\alpha_{V}\circ\rho(x_1,\cdots,\widehat{x}_i,\cdots,x_{n-1},
\alpha^{-1}\beta(x_n)))\alpha_{V}(\beta^{-1}_{V}(u_i)).
\end{eqnarray*}
On the other hand, we have
\begin{eqnarray*}
&&[(\alpha+\alpha_{V})(x_1+u_1),\cdots,(\alpha+\alpha_{V})(x_n+u_n)]_{\rho}\\&=&[\alpha(x_1)+\alpha_{V}(u_1),\cdots,\alpha(x_n)+\alpha_{V}(u_n)]_{\rho}\\
                                                               &=&[\alpha(x_1),\cdots,\alpha(x_n)]_\mathfrak{g}+\sum_{i=1}^{n}(-1)^{n-i}\rho(\alpha(x_1),\cdots,\widehat{\alpha}(x_i),
                                                               \cdots,\alpha^{-1}\beta(\alpha(x_n)))\alpha_{V}(\beta^{-1}_{V}(\alpha_{V}(u_i)))                            \end{eqnarray*}
Since $\alpha$ is an algebra morphism, $\rho$ and $\alpha_{V}$ satisfy the condition (i) in Definition \eqref{defi:rep}, it follows that $\alpha+\alpha_{V}$ is an algebra morphism with respect to the bracket $[\cdot,\cdots,\cdot]_{\rho}$.\\
Similarly, we have $\beta+\beta_{V}$ is an algebra morphism with respect to the bracket $[\cdot,\cdots,\cdot]_{\rho}$.

Next we show that the bracket $[\cdot,\cdots,\cdot]_{\rho}$ satisfies BiHom skewsymmetry, $\forall x_i\in\mathfrak{g},u_i\in V$.
$$\begin{array}{llll}&&[(\beta+\beta_{V})(x_1+u_1),\cdots,(\beta+\beta_{V})(x_{i}+u_{i}),(\beta+\beta_{V})(x_{i+1}+u_{i+1}),\cdots,(\alpha+\alpha_{V})(x_n+u_n)]_{\rho}\\[0.2cm]
&=&[\beta(x_{1})+\beta_{V}(u_{1}),\cdots,\beta(x_{i})+\beta_{V}(u_{i}),\beta(x_{i+1})+\beta_{V}(u_{i+1}),\cdots,\alpha(x_n)+\alpha_{V}(u_n)]_{\rho}
\\&=&[\beta(x_1),\cdots,\beta(x_i),\beta(x_{i+1}),\cdots,\alpha(x_n)]_\mathfrak{g}\\&+&\displaystyle{\sum_{k=1}^{n-1}}(-1)^{n-k}\rho(\beta(x_1),\cdots,\widehat{x}_k,\cdots,\beta(x_{i}),
\beta(x_{i+1}),\beta(x_{n-1}),\alpha^{-1}\beta(\alpha(x_n)))
(\alpha_{V}\beta^{-1}_{V}(\beta_{V}(u_k)))\\[0.2cm]
&+&\rho(\beta(x_1),\cdots,\beta(x_{i}),\beta(x_{i+1}),\cdots,\beta(x_{n-1}))(\alpha_{V}(u_n))
\\&=&-[\beta(x_1),\cdots,\beta(x_{i+1}),\beta(x_i),\cdots,\alpha(x_n)]_\mathfrak{g}\\&-&\displaystyle{\sum_{k=1}^{n-1}}(-1)^{n-k}
\rho(\beta(x_1),\cdots,\widehat{x}_k,\cdots,\beta(x_{i+1}),\beta(x_{i}),\beta(x_{n-1}),\alpha^{-1}\beta(\alpha(x_n)))
(\alpha_{V}\beta^{-1}_{V}(\beta_{V}(u_k)))\\[0.2cm]
&-&\rho(\beta(x_1),\cdots,\beta(x_{i+1}),\beta(x_{i}),\cdots,\beta(x_{n-1}))(\alpha_{V}(u_n))\\[0.2cm]&=&-[\beta(x_{1})+\beta_{V}(u_{1}),\cdots,\beta(x_{i+1})+\beta_{V}(u_{i+1}),
\beta(x_{i})+\beta_{V}(u_{i}),\cdots,\alpha(x_n)+\alpha_{V}(u_n)]_{\rho}\\[0.2cm]&=&-[(\beta+\beta_{V})(x_1+u_1),\cdots,(\beta+\beta_{V})(x_{i+1}+u_{i+1}),(\beta+\beta_{V})(x_{i}+u_{i}),
\cdots,(\alpha+\alpha_{V})(x_n+u_n)]_{\rho}.
\end{array}$$
In the same way, we also have $$\begin{array}{lllll}&&[(\beta+\beta_{V})(x_1+u_1),\cdots,(\beta+\beta_{V})(x_{n-1}+u_{n-1}),(\alpha+\alpha_{V})(x_n+u_n)]_{\rho}\\[0.2cm]
&=&-[(\beta+\beta_{V})(x_1+u_1),\cdots,(\beta+\beta_{V})(x_{n}+u_{n}),(\alpha+\alpha_{V})(x_{n-1}+u_{n-1})]_{\rho}\end{array}$$

Finally, for all $x_{i},y_{i}\in\mathfrak{g},~u_{i},v_{i}\in V$ we have
$$\begin{array}{llllll}&&[(\beta+\beta_{V})^{2}(x_{1}+u_{1}),\cdots,(\beta+\beta_{V})^{2}(x_{n-1}+u_{n-1}),[(\beta+\beta_{V})(y_{1}+v_{1}),\cdots,(\alpha+\alpha_{V})(y_{n}+v_{n})]_{\rho}]_{\rho}\\[0.2cm]
&=&[\beta^{2}(x_{1})+\beta_{V}^{2}(u_{1}),\cdots,\beta^{2}(x_{n-1})+\beta_{V}^{2}(u_{n-1}),[\beta(y_{1})+\beta_{V}(v_{1}),\cdots,\alpha(y_{n})+\alpha_{V}(v_{n})]_{\rho}]_{\rho}\\[0.2cm]
&=&[\beta^{2}(x_{1})+\beta_{V}^{2}(u_{1}),\cdots,\beta^{2}(x_{n-1})+\beta_{V}^{2}(u_{n-1}),[\beta(y_1),\cdots,\alpha(y_n)]_\mathfrak{g}\\[0.2cm]&+&\displaystyle{\sum_{i=1}^{n-1}}(-1)^{n-i}\rho(\beta(y_1),\cdots,\widehat{\beta(y_{i})},\cdots,\alpha(y_{n-1}),\alpha^{-1}\beta(\alpha(y_n)))
(\alpha_{V}\beta^{-1}_{V}(\beta_{V}(v_i)))\\[0.2cm]
&+&\rho(\beta(y_1),\cdots,\beta(y_{n-1}))(\alpha(v_n))]_{\rho}\\[0.2cm]
&=&[\beta^{2}(x_{1}),\cdots,\beta^{2}(x_{n-1}),[\beta(y_1),\cdots,\alpha(y_n)]_\mathfrak{g}]_\mathfrak{g}\\[0.2cm]
&+&\displaystyle{\sum_{i=1}^{n-1}}(-1)^{n-i}\rho(\beta^{2}(x_{1}),\cdots,\widehat{\beta^{2}(x_{i})},\cdots,\beta^{2}(x_{n-1}),\alpha^{-1}\beta([\beta(y_1),\cdots,\alpha(y_n)]
_\mathfrak{g})(\alpha_{V}\beta_{V}(u_{i})\\[0.2cm]
&+&\rho(\beta^{2}(x_{1}),\cdots,\beta^{2}(x_{n-1})(\displaystyle{\sum_{i=1}^{n-1}}(-1)^{n-i}\rho(\beta(y_1),\cdots,\widehat{\beta(y_{i})},\cdots,\beta(y_{n-1}),\alpha^{-1}\beta(\alpha(y_n)))
(\alpha_{V}(v_i)))\\[0.2cm]
&+&\rho(\beta(y_1),\cdots,\beta(y_{n-1}))(\alpha_{V}(v_n)))\\[0.2cm]
&=&\displaystyle{\sum_{i=1}^{n}} (-1)^{n-i}[\beta^{2}(y_{1}),\ldots,\beta^{2}(y_{i},\beta^{2}(y_{n}),[\beta(x_{1}),\ldots, \beta(x_{n-1}),\alpha(y_{i})]_\mathfrak{g}]_\mathfrak{g}\\[0.2cm]
&+&\rho(\beta(y_1),\cdots,\beta(y_{n-1}))(\alpha_{V}(v_n))\\[0.2cm]
&+&\displaystyle{\sum_{i=1}^{n}}(-1)^{n-i}\rho(\beta^{2}(y_1),\cdots,\widehat{\beta^{2}(y_{i})},\cdots,\beta^{2}(y_{n-1}),\beta^{2}(y_n))
\rho(\beta(x_{1}),\cdots,\beta(x_{n-1})(\alpha_{V}((v_i))\\
&+&\displaystyle{\sum_{i=1}^{n-1}}(-1)^{n-i}\rho(\beta^{2}(x_{1}),\cdots,\widehat{\beta^{2}(x_{i})},\cdots,\beta^{2}(x_{n-1}),\alpha^{-1}\beta([\beta(y_1),\cdots,\alpha(y_n)]_\mathfrak{g})(\alpha_{V}\beta_{V}(u_{i}))\\[0.2cm]
&=&\displaystyle{\sum_{k=1}^{n}}(-1)^{n-k}[\beta^{2}(y_{1}),\cdots,\widehat{y_{k}},\cdots,\beta^{2}(y_{n}),[\beta(x_1),\cdots,\beta(x_{n-1}),\alpha(y_{k})]_\mathfrak{g}]_\mathfrak{g}\\[0.2cm]
&+&\displaystyle{\sum_{k=1}^{n}}(-1)^{n-k}\Big(\displaystyle{\sum_{i=1}^{n-1}}(-1)^{n-i}\rho(\beta^{2}(y_{1}),\cdots,\widehat{y_{k}},\cdots,\widehat{\beta^{2}(y_{i})},\cdots,\beta^{2}(y_{n}),\alpha^{-1}\beta([\beta(x_1),\cdots,\\[0.2cm]&&\beta(x_{n-1}),\alpha(y_k)]_\mathfrak{g})(\alpha_{V}\beta_{V}(v_{i}))\Big)\\
&+&\displaystyle{\sum_{k=1}^{n}}(-1)^{n-k}\Big(\rho(\beta^{2}(y_{1}),\cdots,\beta^{2}(y_{k}),\cdots,\beta^{2}(y_{n})(\displaystyle{\sum_{i=1}^{n-1}}(-1)^{n-i}\rho(\beta(x_1),\cdots,\widehat{\beta(x_{i})},\cdots,\\[0.2cm]
&+&\beta(x_{n-1}),\alpha^{-1}\beta(\alpha(y_k)))
(\alpha_{V}(u_i)))\rho(\beta(x_1),\cdots,\beta(x_{n-1}))(\alpha_{V}(v_k)))\Big)\\[0.2cm]
&=&\displaystyle{\sum_{k=1}^{n}}(-1)^{n-k}\Big\{[\beta^{2}(y_{1}),\cdots,\widehat{\beta^{2}(y_{k})},\cdots,\beta^{2}(y_{n}),[\beta(x_1),\beta(x_{n-1}),\cdots,\alpha(y_k)]_\mathfrak{g}]_\mathfrak{g}\\[0.2cm]
&+&\displaystyle{\sum_{i=1}^{n-1}}(-1)^{n-i}\rho(\beta^{2}(y_{1}),\cdots,\widehat{\beta^{2}(y_{k})},\cdots,\widehat{\beta^{2}(y_{i})},\cdots,\beta^{2}(y_{n}),\alpha^{-1}\beta
([\beta(x_1),\cdots,\beta(x_{n-1}),\alpha(y_k)]
_\mathfrak{g})(\alpha_{V}\beta_{V}(v_{i})\\[0.2cm]
&+&\rho(\beta^{2}(x_{1}),\cdots,\widehat{\beta^{2}(x_{k}},\cdots,\beta^{2}(y_{k})(\displaystyle{\sum_{i=1}^{n-1}}(-1)^{n-i}\rho(\beta(x_1),\cdots,\widehat{\beta(x_{i})},\cdots,\beta(x_{n-1}),\alpha^{-1}\beta(\alpha(y_k)))
(\alpha_{V}(u_i)))\\[0.2cm]
&+&\rho(\beta(x_1),\cdots,\beta(x_{n-1}))(\alpha_{V}(v_k)))\Big\}\\[0.2cm]
&=&\displaystyle{\sum_{k=1}^{n}}(-1)^{n-k}[(\beta+\beta_{V})^{2}(y_{1}+v_{1}),\cdots,(\beta+\beta_{V})^{2}(y_{k-1}+v_{k-1}),(\beta+\beta_{V})^{2}(y_{k+1}+v_{k+1}),\cdots\\&&(\beta+\beta_{V})^{2}(y_{n}+v_{n}),[(\beta+\beta_{V})(x_{1}+u_{1}),\cdots,(\beta+\beta_{V})(x_{n-1}+u_{n-1}),(\alpha+\alpha_{V})(y_{k}+v_{k})]_{\rho}]_{\rho}
\end{array}$$
Then $\mathfrak{g}\ltimes V:=(\mathfrak{g}\oplus V,[\cdot,\cdots,\cdot]_{\rho},\alpha+\alpha_{V},\beta+\beta_{V})$ is an $n$-BiHom-Lie algebra.
\end{proof}
\begin{defn}
  Let $(V_1;\rho_1,\alpha_1,\beta_1)$ and $(V_2;\rho_2,\alpha_2,\beta_2)$ be two representations of an $n$-BiHom-Lie algebra $(\mathfrak{g},[\cdot,\cdots,\cdot]_\mathfrak{g},\alpha,\beta)$. They are said to be {\bf equivalent} if there exists an isomorphism of vector spaces $T:V_1\longrightarrow V_2$ such that
  $$
  T\rho_1(X)(u)=\rho_2(X)(Tu),\quad T\circ \alpha_1=\alpha_2\circ T,\quad T\circ \beta_1=\beta_2\circ T,\quad\forall X=x_1\wedge\cdots\wedge x_{n-1},~u\in V_1.
  $$
  In terms of diagrams, we have
  $$
\xymatrix{
 \wedge^{n-1}\mathfrak{g}\times V_1 \ar[d]_{ id\times T }\ar[rr]^{\rho_1}
                && V_1  \ar[d]^{T}  \\
 \wedge^{n-1}\mathfrak{g}\times V_2 \ar[rr]^{\rho_2}
                && V_2  },\quad \xymatrix{
 V_1 \ar[d]_{ T }\ar[rr]^{\alpha_1}
                && V_1  \ar[d]^{T}  \\
 V_2 \ar[rr]^{\alpha_2}
                && V_2  },\quad \xymatrix{
 V_1 \ar[d]_{ T }\ar[rr]^{\beta_1}
                && V_1  \ar[d]^{T}  \\
 V_2 \ar[rr]^{\beta_2}
                && V_2.  }
$$
\end{defn}
\begin{thm}
Let $(\mathfrak{g},[\cdot,\cdots,\cdot]_\mathfrak{g})$ be an $n$-Lie algebra and $(V,\rho)$ be a representation of $\mathfrak{g}$. Let $\alpha,\beta:\mathfrak{g}\rightarrow\mathfrak{g}$ be two endomorphisms of $\mathfrak{g}$ and any two of the maps $\alpha,\beta$ commute, and let $\alpha_{V},\beta_{V}:V\rightarrow V$ be two linear maps of $V$ and any two of the maps $\alpha_{V},\beta_{V}$ commute such that $\alpha_{V}\circ\rho(X)=\rho(\widetilde{\alpha}(X))\circ \alpha_{V}$ and $\beta_{V}\circ\rho(X)=\rho(\widetilde{\beta}(X))\circ \beta_{V}$. Then $(V,\widetilde{\rho}:=\beta_{V}\circ\rho,\alpha_{V},\beta_{V})$ is a representation of the $n$-BiHom-Lie algebra $(\mathfrak{g},[\cdot,\cdots,\cdot]_{\alpha,\beta},[\cdot,\cdots,\cdot]_\mathfrak{g}\circ(\alpha\otimes\cdots\otimes\alpha\otimes\beta),\alpha,\beta)$.
\end{thm}
\begin{proof}
For all $X=x_1\wedge\cdots\wedge x_{n-1},~Y=y_1\wedge\cdots\wedge y_{n-1}\in\wedge^{n-1}\mathfrak{g}$, we have
\begin{enumerate}
\item
\begin{eqnarray*}
\widetilde{\rho}(\widetilde{\alpha}(X) )\circ\alpha_{V}&=&\beta_{V}\circ\rho(\widetilde{\alpha}(X) )\circ\alpha_{V}\\
&=&\beta_{V}\alpha_{V}\circ \rho(X)\\&=&\alpha_{V}\beta_{V}\circ \rho(X)\\
&=&\alpha_{V}\circ \widetilde{\rho}(X)\end{eqnarray*}
\item
\begin{eqnarray*}
\widetilde{\rho}(\widetilde{\beta}(X) )\circ\beta_{V}&=&\beta_{V}\circ\rho(\widetilde{\beta}(X) )\circ\beta_{V}\\
&=&\beta_{V}^{2}\circ \rho(X)\\
&=&\beta_{V}\circ \widetilde{\rho}(X)\end{eqnarray*}
\item
\begin{eqnarray*} &&\widetilde{\rho}(\widetilde{\alpha\beta}(X))\circ\widetilde{\rho}(Y)-\widetilde{\rho}(\widetilde{\beta}(Y))\circ \widetilde{\rho}(\widetilde{\alpha}(X))\\&-&\sum_{i=1}^{n-1}\widetilde{\rho}(\beta(y_1),\cdots,\beta(y_{i-1}),[\beta(x_1),\cdots,\beta(x_{n-1}),y_i]_{\alpha\beta},\beta(y_{i+1})
 ,\cdots,\beta(y_{n-1}))\circ \beta_{V}\\
&=&\beta_{V}\rho(\widetilde{\alpha\beta}(X))\circ\beta_{V}\rho(Y)-\beta_{V}\rho(\widetilde{\beta}(Y))\circ \beta_{V}\rho(\widetilde{\alpha}(X))\\&-&\sum_{i=1}^{n-1}\beta_{V}\rho(\beta(y_1),\cdots,\beta(y_{i-1}),[\alpha\beta(x_1),\cdots,\alpha\beta(x_{n-1}),\beta(y_i)]_{\mathfrak{g}},\beta(y_{i+1})
 ,\cdots,\beta(y_{n-1}))\circ \beta_{V}\\
&=&\beta_{V}^{2}\rho(\widetilde{\alpha}(X))\circ\rho(Y)-\beta_{V}^{2}\rho(Y)\circ \rho(\widetilde{\alpha}(X))\\&-&\sum_{i=1}^{n-1}\beta_{V}^{2}\rho(y_1,\cdots,y_{i-1},[\alpha(x_1),\cdots,\alpha(x_{n-1}),y_i]_{\mathfrak{g}},y_{i+1}
 ,\cdots,y_{n-1})\\
&=&\beta_{V}^{2}\Big(\rho(\widetilde{\alpha}(X))\circ\rho(Y)-\rho(Y)\circ \rho(\widetilde{\alpha}(X))\\&-&\sum_{i=1}^{n-1}\rho(y_1,\cdots,y_{i-1},[\alpha(x_1),\cdots,\alpha(x_{n-1}),y_i]_{\mathfrak{g}},y_{i+1}
 ,\cdots,y_{n-1})\Big)=0
\end{eqnarray*}
 \item\begin{eqnarray*}
        &&\widetilde{\rho}(\beta(y_2),\cdots,\beta(y_{n-1}),[\beta(x_1),\cdots,\beta(x_{n-1}),y_1]_{\alpha\beta})\circ\beta_{V}\\
        &-&\sum_{i=1}^{n-1}(-1)^{n-i}\widetilde{\rho}(\alpha\beta(x_1),\cdots,\widehat{\alpha\beta}(x_i),\cdots,\alpha\beta(x_{n-1}),\beta(y_1))\circ
\widetilde{\rho}(y_2,\cdots,y_{n-1},\alpha(x_i))\\
        &-&\widetilde{\rho}(\widetilde{\alpha\beta}(X))\circ\widetilde{\rho}(Y)\\
&=&\beta_{V}\rho(\beta(y_2),\cdots,\beta(y_{n-1}),[\alpha\beta(x_1),\cdots,\alpha\beta(x_{n-1}),\beta(y_1)]_{\mathfrak{g}})\circ\beta_{V}\\
        &-&\sum_{i=1}^{n-1}(-1)^{n-i}\beta_{V}\rho(\alpha\beta(x_1),\cdots,\widehat{\alpha\beta}(x_i),\cdots,\alpha\beta(x_{n-1}),\beta(y_1))\circ
\beta_{V}\rho(y_2,\cdots,y_{n-1},\alpha(x_i))\\
        &-&\beta_{V}\rho(\widetilde{\alpha\beta}(X))\circ\beta_{V}\rho(Y)\\
&=&\beta_{V}^{2}\rho(y_2,\cdots,y_{n-1},[\alpha(x_1),\cdots,\alpha(x_{n-1}),y_1]_{\mathfrak{g}})\\
        &-&\sum_{i=1}^{n-1}(-1)^{n-i}\beta_{V}^{2}\rho(\alpha(x_1),\cdots,\widehat{\alpha}(x_i),\cdots,\alpha(x_{n-1}),y_1)\rho(y_2,\cdots,y_{n-1},\alpha(x_i))\\
        &-&\beta_{V}^{2}\rho(\widetilde{\alpha}(X))\circ\rho(Y)\\
&=&\beta_{V}^{2}\Big(\rho(y_2,\cdots,y_{n-1},[\alpha(x_1),\cdots,\alpha(x_{n-1}),y_1]_{\mathfrak{g}})\\
        &-&\sum_{i=1}^{n-1}(-1)^{n-i}\rho(\alpha(x_1),\cdots,\widehat{\alpha}(x_i),\cdots,\alpha(x_{n-1}),y_1)\rho(y_2,\cdots,y_{n-1},\alpha(x_i))\\
        &-&\rho(\widetilde{\alpha}(X))\circ\rho(Y)\Big)=0.
        \end{eqnarray*}

\end{enumerate}
\end{proof}
\begin{ex}
Let $\mathfrak{g}$ be the $3$-dimensional $3$-Lie algebra defined with respect to a  basis $\{e_1,e_2,e_3\}$ by the skew-symmetric bracket   $[e_1,e_2,e_3]_{\mathfrak{g}}= e_1$.
 Let $V$ be a $2$-dimensional vector space and $\{v_1,v_2\}$ its basis. We have a representation defined by  the map $\rho$, given with respect to previous bases by
\begin{eqnarray*}    &\rho(e_1,e_2)(v_1)=0,  \;
\rho(e_1,e_2)(v_2)= v_1 ,   \;
 \rho(e_1,e_3)(v_1)=0, \\&
  \rho(e_1,e_3)(v_2)=0, \;
 \rho(e_2,e_3)(v_1)= v_1,    \;  \rho(e_{2},e_3)(v_2)=0,
\end{eqnarray*}
Let $\alpha,\beta: \mathfrak{g}\longrightarrow \mathfrak{g}$ be  a algebra morphism and $\alpha_V,\beta_V \in gl (V)$ defined respectively by:
\begin{align*}
&\alpha(e_1)=- e_1, \alpha(e_2)=e_2, \alpha(e_3)=e_3,\\
&\beta(e_1)=-e_1, \beta(e_2)=-e_2, \beta(e_3)=e_3,\\
& \alpha_V(v_1)=- v_1, \alpha_V(v_2)=v_2,\\& \beta_V(v_1)=v_1, \beta_V(v_2)=-v_2,
\end{align*}
they satisfy
\begin{align*}
&\alpha\beta=\beta\alpha,~~~~\alpha_V\beta_V=\beta_V\alpha_V,\\
& \alpha_{V}\circ \rho(x_{1},x_{2})= \rho(\alpha(x_{1}),\alpha(x_{2}))\circ \alpha_V,\\
& \beta_{V}\circ \rho(x_{1},x_{2})= \rho(\beta(x_{1}),\beta(x_{2}))\circ \beta_V.
 \end{align*}
where $x_1,\cdots,x_{n-1}$ are in $\mathfrak{g}$.

Then, using the Twist procedure, $(V;\widetilde{\rho}=\beta_{V}\circ\rho,\alpha_V,\beta_V)$ is a representation of the $n$-BiHom-Lie algebra $(\mathfrak{g}, [\cdot,\cdots,\cdot]_{\alpha,\beta},\alpha,\beta)$. More precisely, we have
\begin{eqnarray*}
  [e_1,e_2,e_3]_{\alpha,\beta}&=&[\alpha(e_1),\alpha(e_2),\beta(e_3)]_\mathfrak{g}=- e_1,
  \end{eqnarray*}
  \begin{eqnarray*}
 &  \widetilde{\rho}(e_1,e_2)(v_1)=0,  \;
\widetilde{\rho}(e_1,e_2)(v_2)=v_1 ,   \;
 \widetilde{\rho}(e_1,e_3)(v_1)=0, \\
&  \widetilde{\rho}(e_1,e_3)(v_2)=0, \;
 \widetilde{\rho}(e_2,e_3)(v_1)=v_1,    \;  \widetilde{\rho}(e_2,e_3)(v_2)=0.
\end{eqnarray*}
\end{ex}
\begin{defn}
Let $(\mathfrak{g},[\cdot,\cdots,\cdot]_{\mathfrak{g}},\alpha,\beta)$ be an $n$-BiHom-Lie algebra. A bilinear form  $f$ on $\mathfrak{g}$ is said to be nondegenerate if
$$\mathfrak{g}^\perp=\{x\in \mathfrak{g}~|~f(x,y)=0, ~\forall\, y\in \mathfrak{g}\}=0;$$
$\alpha\beta$-invariant if for all $x_1,\cdots,x_{n+1}\in \mathfrak{g}$,
$$f([\beta(x_1),\cdots,\beta(x_{n-1}),\alpha(x_n)]_{\mathfrak{g}},\alpha(x_{n+1}))=-f(\alpha(x_{n}),[\beta(x_1),\cdots,\beta(x_{n-1}),\alpha(x_{n+1})]_{\mathfrak{g}});$$
symmetric if
$$f(x,y)=f(y,x).$$
$\alpha$ is called $f$-symmetric, if $f(\alpha(x),y)=f(x,\alpha(y))$.\\
A subspace $I$ of $\mathfrak{g}$ is called isotropic if $I\subseteq I^\bot$.
\end{defn}

\begin{defn}
Let $(\mathfrak{g},[\cdot,\cdots,\cdot]_{\mathfrak{g}},\alpha,\beta)$ be an $n$-BiHom-Lie algebra over a field $\mathbb{K}$. If $\mathfrak{g}$ admits a nondegenerate, $\alpha\beta$-invariant and symmetric bilinear form $f$ such that $\alpha$, $\beta$ are $f$-symmetric, then we call $(\mathfrak{g},f,\alpha,\beta)$ a quadratic $n$-BiHom-Lie algebra.

Let $(\mathfrak{g}^{'},[\cdot,\cdots,\cdot]_{\mathfrak{g}'},\alpha',\beta')$ be another $n$-BiHom-Lie algebra. Two quadratic $n$-BiHom-Lie algebras $(\mathfrak{g},f,\alpha,\beta)$ and
$(\mathfrak{g}^{'},f',\alpha',\beta')$ are said to be isometric if there exists a algebra isomorphism $\phi: \mathfrak{g}\rightarrow \mathfrak{g}^{'}$ such that
$f(x, y)=f'(\phi(x), \phi(y)),~ \forall\, x, y\in \mathfrak{g}$.
\end{defn}

\begin{thm}
Let $(\mathfrak{g},[\cdot,\cdots,\cdot]_{\mathfrak{g}},\alpha,\beta)$ be an $n$-BiHom-Lie algebra and $(V, \rho, \alpha_{V}, \beta_V)$ be a representation of $\mathfrak{g}$. Let us consider $V^*$ the dual space of $V$ and $\tilde\alpha_{V},\tilde\beta_V:V^*\rightarrow V^*$ two Homomorphisms defined by $
\tilde\alpha_{V}(f)=f\circ\alpha_{V},\tilde\beta_V(f)=f\circ\beta_V,~ \forall\, f\in V^*$. Then the skewsymmetry linear map $\tilde\rho:\wedge^{n-1}\mathfrak{g}\rightarrow\mathrm{End}(V^{*})$, defined by $\tilde\rho(x_{1},\cdots,x_{n-1})(f)=- f\circ \rho(x_{1},\cdots,x_{n-1}),~ \forall\, f \in V^*, x_{1},\cdots,x_{n-1} \in \mathfrak{g},$ is a representation of $\mathfrak{g}$ on $(V^{*},\tilde\rho,\tilde{\alpha}_V,\tilde\beta_V)$ if and only if for every $X=x_1\wedge\cdots\wedge x_{n-1},~Y=y_1\wedge\cdots\wedge y_{n-1}\in\wedge^{n-1}\mathfrak{g}$,
\begin{enumerate}
\item
$\alpha_{V}\circ \rho(\widetilde{\alpha}(X))=\rho(X)\circ \alpha_{V},$
\item $\beta_V\circ \rho(\widetilde{\beta}(X))=\rho(X)\circ \beta_V,$

  \item \begin{eqnarray*}&&\rho(Y)\circ\rho(\widetilde{\alpha\beta}(X))-\rho(\widetilde{\alpha}(X))\circ\rho(\widetilde{\beta}(Y))\\ &=&-\sum_{i=1}^{n-1}\beta_{V}\rho(\beta(y_1),\cdots,\beta(y_{i-1}),[\beta(x_1),\cdots,\beta(x_{n-1}),y_i]_\mathfrak{g},\beta(y_{i+1})
 ,\cdots,\beta(y_{n-1})));\end{eqnarray*}

    \item\begin{eqnarray*}
        &&\beta_{V}\rho([\beta(x_1),\cdots,\beta(x_{n-1}),y_1]_\mathfrak{g},\beta(y_2),\cdots,\beta(y_{n-1}),)\\
       &=&-\sum_{i=1}^{n-1}(-1)^{n-i}\rho(y_2,\cdots,y_{n-1},\alpha(x_i))\circ\rho(\alpha\beta(x_1),\cdots,\widehat{\alpha\beta}(x_i),\cdots,\alpha\beta(x_{n-1}),\beta(y_1))\\
        &-&\rho(Y)\circ\rho(\widetilde{\alpha\beta}(X)).
        \end{eqnarray*}
\end{enumerate}
\end{thm}

\begin{proof}Let $f \in V^*$, $X=x_1\wedge\cdots\wedge x_{n-1},~Y=y_1\wedge\cdots\wedge y_{n-1}\in\wedge^{n-1}\mathfrak{g}$.
First, we have
$$
(\tilde\rho(\alpha(x_1),\cdots,\alpha(x_{n-1}))\circ \tilde{\alpha}_V)(f)=- \tilde{\alpha}_V(f)\circ \rho(\alpha(x_{1}),\cdots,\alpha(x_{n-1}))=-f\circ \alpha_{V}\circ \rho(\alpha(x_{1}),\cdots,\alpha(x_{n-1}))$$ and
$\tilde{\alpha}_V\circ\tilde\rho(x_{1},\cdots,x_{n-1})(f)=-\tilde{\alpha}_V(f\circ\rho(x_{1},\cdots,x_{n-1}))=-f\circ\rho(x_{1},\cdots,x_{n-1})\circ\alpha_{V},$ which implies $$\tilde\rho(\alpha(x_{1}),\cdots,\alpha(x_{n-1}))\circ \tilde{\alpha}_V=\tilde{\alpha}_V\circ\tilde\rho(x_1,\cdots,x_{n-1})\Leftrightarrow \alpha_{V}\circ \rho(\alpha(x_{1}),\cdots\alpha(x_{n-1}))=\rho(x_1,\cdots,x_{n-1})\circ\alpha_{V}.$$
Similarly, $\tilde\rho(\beta(x_1),\cdots,\beta(x_{n-1}))\circ \tilde{\beta}_V=\tilde{\beta}_V\circ\tilde\rho(x_1,\cdots,x_{n-1})\Leftrightarrow \beta_V\circ \rho(\beta(x_1),\cdots,\beta(x_{n-1}))=\rho(x_1,\cdots,x_{n-1})\circ\beta_V.$

Then we can get
$$
\begin{array}{llll}
\tilde\rho(\alpha\beta(x_{1}),\cdots,\alpha\beta(x_{n-1}))\circ\tilde\rho(y_{1},\cdots,y_{n-1})(f)
&=&-\tilde\rho(\alpha\beta(x_1),\cdots,\alpha\beta(x_{n-1}))(f\rho(y_1,\cdots,y_{n-1}))\\[0.2cm]
&=&f\rho(y_1,\cdots,y_{n-1})\rho(\alpha\beta(x_1),\cdots,\alpha\beta(x_{n-1}))
\end{array}$$ and
\begin{eqnarray*}
&&\big(\tilde\rho(\beta(y_1),\cdots,\beta(y_{n-1}))\circ\tilde\rho(\alpha(x_1),\cdots,\alpha(x_{n-1}))\\&+&\sum_{i=1}^{n-1}\tilde{\rho}(\beta(y_1),\cdots,\beta(y_{i-1}),
[\beta(x_1),\cdots,\beta(x_{n-1}),y_i]_\mathfrak{g},\beta(y_{i+1})
,\cdots,\beta(y_{n-1})\circ\tilde\beta_V\big)(f)\\
&=&\!f\rho(\alpha(x_{1}),\cdots,\alpha(x_{n-1})\!)\!\rho(\beta(y_1),\cdots,\beta(y_{n-1})\!)\\&-&
\!\sum_{i=1}^{n-1}f\beta_{V}\rho(\beta(y_1),\cdots,\beta(y_{i-1}),[\beta(x_1),\cdots,\beta(x_{n-1}),y_i]_\mathfrak{g},\beta(y_{i+1})
 ,\cdots,\beta(y_{n-1}),
\end{eqnarray*}
which implies
\begin{eqnarray*}
&&\tilde\rho(\alpha\beta(x_1),\cdots,\alpha\beta(x_{n-1}))\circ\tilde\rho(y_1,\cdots,y_{n-1})\\
&=&\tilde\rho(\beta(y_1),\cdots,\beta(y_{n-1}))\!\circ\!\tilde\rho(\alpha(x_1),\cdots,\alpha(x_{n-1})\!)\\&+&\sum_{i=1}^{n-1}\tilde{\rho}(\beta(y_1),\cdots,\beta(y_{i-1}),
[\beta(x_1),\cdots,\beta(x_{n-1}),y_i]_\mathfrak{g},\beta(y_{i+1})
,\cdots,\beta(y_{n-1})\!\circ\!\tilde\beta_V
\end{eqnarray*}
if and only if
\begin{eqnarray*}
&&\rho(y_{1},\cdots,y_{n-1})\rho(\alpha\beta(x_{1}),\cdots,\alpha\beta(x_{n-1}))\\
&=&\rho(\alpha(x_{1}),\cdots,\alpha(x_{n-1}))\rho(\beta(x_1),\cdots,\beta(x_{n-1}))\\&-&\sum_{i=1}^{n-1}\beta_{V}\rho(\beta(y_1),\cdots,\beta(y_{i-1}),[\beta(x_1),\cdots,\beta(x_{n-1}),y_i]_\mathfrak{g},\beta(y_{i+1})
 ,\cdots,\beta(y_{n-1}))).
\end{eqnarray*}
In the same way,
\begin{eqnarray*}
&&\tilde\rho([\beta(x_{1}),\cdots,\beta(x_{n-1}),y_1]_\mathfrak{g},\beta(y_2),\cdots,\beta(y_{n-1}))\tilde\beta_V\\
 &=&\sum_{i=1}^{n-1}(-1)^{n-i}\tilde{\rho}(\alpha\beta(x_1),\cdots,\widehat{\alpha\beta}(x_i),\cdots,\alpha\beta(x_{n-1}),\beta(y_1))\circ\tilde{\rho}(y_2,\cdots,y_{n-1},\alpha(x_i))\\
        &+&\tilde{\rho}(\alpha\beta(x_1),\cdots,\alpha\beta(x_{n-1}))\circ\tilde{\rho}(y_1,\cdots,y_{n-1}).
\end{eqnarray*}
if and only if
\begin{eqnarray*}
 &&\beta_{V}\rho([\beta(x_1),\cdots,\beta(x_{n-1}),y_1]_\mathfrak{g}\beta(y_2),\cdots,\beta(y_{n-1}))\\
       &=&-\sum_{i=1}^{n-1}(-1)^{n-i}\rho(y_2,\cdots,y_{n-1},\alpha(x_i))\circ\rho(\alpha\beta(x_1),\cdots,\widehat{\alpha\beta}(x_i),\cdots,\alpha\beta(x_{n-1}),\beta(y_1))\\
        &-&\rho(y_1,\cdots,y_{n-1})\circ\rho(\alpha\beta(x_1),\cdots,\alpha\beta(x_{n-1})).
\end{eqnarray*}

That shows the theorem holds.
\end{proof}

\begin{cor}
Let $\ad$ be the adjoint representation of an $n$-BiHom-Lie algebra $(\mathfrak{g},[\cdot,\cdots,\cdot]_{\mathfrak{g}},\\
\alpha,\beta)$. Let us consider the bilinear map $\ad^*:\wedge^{n-1}\mathfrak{g}\rightarrow \mathrm{End}(\mathfrak{g}^{*})$ defined by $$\ad^*(x_{1},\cdots,x_{n-1})(f)=-f\circ \ad(x_{1},\cdots,x_{n-1}),~\forall\, x_{1},\cdots,x_{n-1}\in \mathfrak{g}.$$ Then $\ad^*$ is a representation of $\mathfrak{g}$ on $(\mathfrak{g}^{*},\ad^*,\overline{\alpha},\overline\beta)$ if and only if
\begin{enumerate}
\item
$\alpha\circ \ad(\widetilde{\alpha}(X))=\ad(X)\circ \alpha,$
\item $\beta\circ \ad(\widetilde{\beta}(X))=\ad(X)\circ \beta,$

  \item \begin{eqnarray*} &&\ad(Y)\circ \ad(\widetilde{\alpha\beta}(X))-\ad(\widetilde{\alpha}(X))\circ \ad(\widetilde{\beta}(Y))\\ &=&-\sum_{i=1}^{n-1}\beta\ad(\beta(y_1),\cdots,\beta(y_{i-1}),[\beta(x_1),\cdots,\beta(x_{n-1}),y_i]_\mathfrak{g},\beta(y_{i+1})
 ,\cdots,\beta(y_{n-1})));\end{eqnarray*}

    \item\begin{eqnarray*}
         &&\beta\ad([\beta(x_1),\cdots,\beta(x_{n-1}),y_1]_\mathfrak{g},\beta(y_2),\cdots,\beta(y_{n-1}),)\\
       &=&-\sum_{i=1}^{n-1}(-1)^{n-i}\ad(y_2,\cdots,y_{n-1},\alpha(x_i))\circ\ad(\alpha\beta(x_1),\cdots,\widehat{\alpha\beta}(x_i),\cdots,\alpha\beta(x_{n-1}),\beta(y_1))\\
        &-&\ad(Y)\circ\ad(\widetilde{\alpha\beta}(X)).
        \end{eqnarray*}
\end{enumerate}
We call the representation $\ad^*$ the coadjoint representation of $\mathfrak{g}$.
\end{cor}

\section{Extensions of $n$-BiHom-Lie algebras}
The method of $T_\theta$-extension and $T_\theta^*$-extension were introduced in \cite{B} and has already been used for $3$-BiHom-Lie algebras in \cite{Li C}. Now we will
generalize it to $n$-BiHom-Lie algebras.
\subsection{$T_\theta$-extensions of $n$-BiHom-Lie algebras}
\begin{defn}
Let $(\mathfrak{g},[\cdot,\cdots,\cdot]_{\mathfrak{g}},\alpha,\beta)$ be an $n$-BiHom-Lie algebra and $(V,\rho,\alpha_{V},\beta_{V})$ be a representation of $\mathfrak{g}$. If $\theta:
\mathfrak{g}\times\cdots\times\mathfrak{g}\rightarrow V$ is an $n$-linear map and satisfies
\begin{enumerate}
\item
$\alpha_{V}\theta(x_{1},\cdots,x_{n})=\theta(\alpha(x_{1}),\cdots,\alpha(x_{n})),$
\item
$\beta_{V}\theta(x_{1},\cdots,x_{n})=\theta(\beta(x_{1}),\cdots,\beta(x_{n})),$
\item
$\theta(\beta(x_{1}),\cdots,\beta(x_{i}),\beta(x_{i+1}),\cdots,\beta(x_{n}))=-\theta(\beta(x_{1}),\cdots,\beta(x_{i+1}),\beta(x_{i}),\cdots,\beta(x_{n}))=-
\theta(\beta(x_{1}),\cdots,\beta(x_{n-2}),\beta(x_{n}),\beta(x_{n-1})),$
\item
$\theta(\beta^{2}(x_{1}),\cdots,\beta^{2}(x_{n-1}),[\beta(y_{1}),\cdots,\alpha(y_{n})]_{\mathfrak{g}})
+\rho(\beta^{2}(x_{1}),\cdots,\beta^{2}(x_{n-1}))\theta(\beta(y_{1}),\cdots,\alpha(y_{n}))$\\
$\begin{array}{lllll}
&=& \displaystyle{\sum_{k=1}^{n}} (-1)^{n-k} \theta(\beta^{2}(y_{1}),\ldots,\beta^{2}(y_{k-1}),\beta^{2}(y_{k+1}),\ldots,\beta^{2}(y_{n}),[\beta(x_{1}),\ldots, \beta(x_{n-1}),\alpha(y_{k})]_{\mathfrak{g}})\\
&+& \displaystyle{\sum_{k=1}^{n}} (-1)^{n-k} \rho(\beta^{2}(y_{1}),\ldots,\beta^{2}(y_{k-1}),\beta^{2}(y_{k+1}),\ldots,\beta^{2}(y_{n}))\theta(\beta(x_{1}),\ldots, \beta(x_{n-1}),\alpha(y_{k}))\\
\end{array}$
\end{enumerate}
where $x_{i},~y_{i}\in\mathfrak{g}$. Then $\theta$ is called an $n$-cocycle associated with $\rho$.
\end{defn}

\begin{prop}\label{prop1}
Let $(\mathfrak{g},[\cdot,\cdots,\cdot]_{\mathfrak{g}},\alpha,\beta)$ be an $n$-BiHom-Lie algebra and $(V,\rho,\alpha_{V},\beta_{V})$ a representation of $\mathfrak{g}$. Assume that the maps $\alpha$ and $\beta_{V}$ are bijective. If $\theta$ is an $n$-cocycle associated with $\rho$. Then $(\mathfrak{g}\oplus V,[\cdot,\cdots,\cdot]_{\theta},\alpha+\alpha_{V},\beta+\beta_{V})$ is a $n-$BiHom-Lie algebra, where $\alpha+\alpha_{V};~\beta+\beta_{V}:\mathfrak{g}\oplus V\longrightarrow\mathfrak{g}\oplus V$ are defined by $(\alpha+\alpha_{V})(u+x)=\alpha(u)+\alpha_{V}(x)$ and $(\beta+\beta_{V})(u+x)=\beta(u)+\beta_{V}(x),$ and the bracket
$[\cdot,\cdots,\cdot]_{\theta}$ is defined by
$$\begin{array}{llll}[x_1+u_1,\cdots,x_n+u_n]_{\theta}&=&[x_1,\cdots,x_n]_\mathfrak{g}+\theta(x_1,\cdots,x_n)\\
&+&\displaystyle{\sum_{i=1}^{n-1}}(-1)^{n-i}\rho(x_1,\cdots,\widehat{x}_i,\cdots,x_{n-1},\alpha^{-1}\beta(x_n))
(\alpha_{V}\beta^{-1}_{V}(u_i))\\
&+&\rho(x_1,\cdots,x_{n-1})(u_n)
,\,\,\,\,\forall x_i\in\mathfrak{g},u_i\in V.\end{array}$$ $(\mathfrak{g}\oplus V,[\cdot,\cdots,\cdot]_{\theta},\alpha+\alpha_{V},\beta+\beta_{V})$ is called the $T_{\theta}$-extension of $(\mathfrak{g},[\cdot,\cdots,\cdot]_{\mathfrak{g}},\alpha,\beta)$ by $V$, denoted by $T_{\theta}(\mathfrak{g}).$
\end{prop}
\begin{proof}
The proof is similar to Proposition \ref{prop11}.
\end{proof}
\begin{prop}\label{propk}
Let $(\mathfrak{g},[\cdot,\cdots,\cdot]_{\mathfrak{g}},\alpha,\beta)$ be an $n$-BiHom-Lie algebra and $(V,\rho,\alpha_{V},\beta_{V})$ be a representation of $\mathfrak{g}$.
Assume that the maps $\alpha$ and $\beta$ are bijective. $f:\mathfrak{g}\rightarrow V$ is a linear map such that $f\circ\alpha=\alpha_{V}\circ f$ and
$f\circ\beta=\beta_{V}\circ f$. Then the $n$-linear map $\theta_{f}:\mathfrak{g}\times\cdots\times\mathfrak{g}\rightarrow V$ given by
$$\begin{array}{llll}\theta_{f}(x_1,\cdots,x_n)&=&f([x_1,\cdots,x_n]_\mathfrak{g})\\
&-&\displaystyle{\sum_{i=1}^{n-1}}(-1)^{n-i}\rho(x_1,\cdots,\widehat{x}_i,\cdots,x_{n-1},\alpha^{-1}\beta(x_n))
f(\alpha\beta^{-1}(x_i))\\
&-&\rho(x_1,\cdots,x_{n-1})f(x_n),\end{array}$$
$\forall x_i\in\mathfrak{g},$ is a $n-$cocycle associated with $\rho$.
\begin{proof}
$\forall x_{1},\cdots,x_{n}\in \mathfrak{g}$, we have
$$\begin{array}{llll}&&\theta_{f}(\alpha(x_{1}),\cdots,\alpha(x_{n}))\\&=&f([\alpha(x_1),\cdots,\alpha(x_n)]_\mathfrak{g})\\
&-&\displaystyle{\sum_{i=1}^{n-1}}(-1)^{n-i}\rho(\alpha(x_1),\cdots,\widehat{\alpha(x_i)},\cdots,\alpha(x_{n-1}),\alpha^{-1}\beta(\alpha(x_n)))
f(\alpha\beta^{-1}(\alpha(x_i)))\\
&-&\rho(\alpha(x_1),\cdots,\alpha(x_{n-1}))f(\alpha(x_n))\\
&=&\alpha_{V}f([x_1,\cdots,x_n]_\mathfrak{g})\\
&-&\displaystyle{\sum_{i=1}^{n-1}}(-1)^{n-i}\alpha_{V}\rho(x_1,\cdots,\widehat{x}_i,\cdots,x_{n-1},\alpha^{-1}\beta(x_n))
f(\alpha\beta^{-1}(x_i))\\
&-&\alpha_{V}\rho(x_1,\cdots,x_{n-1})f(x_n)\\
&=&\alpha_{V}\theta_{f}(x_{1},\cdots,x_{n}).
\end{array}$$
We also have $\theta_{f}(\beta(x_{1}),\cdots,\beta(x_{n}))=\beta_{V}\theta_{f}(x_{1},\cdots,x_{n})$.

Next, for all $i=1,\cdots,n-2$
$$\begin{array}{llll}&&\theta_{f}(\beta(x_{1}),\cdots,\beta(x_{i}),\beta(x_{i+1}),\cdots,\beta(x_{n-1}),\alpha(x_{n}))\\&=&f([\beta(x_1),\cdots,\beta(x_{i}),\beta(x_{i+1}),\beta(x_{n-1}),\alpha(x_n)]_\mathfrak{g})\\
&-&\displaystyle{\sum_{k=1}^{n-1}}(-1)^{n-k}\rho(\beta(x_1),\cdots,\widehat{\beta(x_k)},\cdots,\beta(x_{i}),\beta(x_{i+1}),\cdots,\beta(x_{n-1}),\alpha^{-1}\beta(\alpha(x_n)))
f(\alpha\beta^{-1}(\beta(x_k)))\\
&-&\rho(\beta(x_1),\cdots,\beta(x_{i}),\beta(x_{i+1}),\cdots,\beta(x_{n-1}))f(\alpha(x_n))\\
&=&-f([\beta(x_1),\cdots,\beta(x_{i+1}),\beta(x_{i}),\beta(x_{n-1}),\alpha(x_n)]_\mathfrak{g})\\
&+&\displaystyle{\sum_{k=1}^{n-1}}(-1)^{n-k}\rho(\beta(x_1),\cdots,\widehat{\beta(x_k)},\cdots,\beta(x_{i+1}),\beta(x_{i}),\cdots,\beta(x_{n-1}),\alpha^{-1}\beta(\alpha(x_n)))
f(\alpha\beta^{-1}(\beta(x_k)))\\
&+&\rho(\beta(x_1),\cdots,\beta(x_{i+1}),\beta(x_{i}),\cdots,\beta(x_{n-1}))f(\alpha(x_n))\\
&=&-\theta_{f}(\beta(x_{1}),\cdots,\beta(x_{i+1}),\beta(x_{i}),\cdots,\beta(x_{n-1}),\beta(x_{n}))
\end{array}$$
Similarly, we can get $\theta_{f}(\beta(x_{1}),\cdots,\beta(x_{n-1}),\alpha(x_{n}))
=-\theta_{f}(\beta(x_{1}),\cdots,\beta(x_{n}),\alpha(x_{n-1}))$.

Finally, $\forall x_1,\cdots,x_{n-1},y_1,\cdots,y_n\in \mathfrak{g}$, we have
$$\begin{array}{llll}&&\theta_{f}(\beta^{2}(x_{1}),\cdots,\beta^{2}(x_{n-1}),[\beta(y_{1}),\cdots,\beta(y_{n-1}),\alpha(y_{n})]_{\mathfrak{g}})\\&+&\rho(\beta^{2}(x_{1}),\cdots,\beta^{2}(x_{n-1}))\theta_{f}(\beta(y_{1}),\cdots,\alpha(y_{n}))\\
&=&f([\beta^{2}(x_{1}),\cdots,\beta^{2}(x_{n-1}),[\beta(y_{1}),\cdots,\beta(y_{n-1}),\alpha(y_{n})]_{\mathfrak{g}}]_\mathfrak{g})\\
&-&\displaystyle{\sum_{i=1}^{n-1}}(-1)^{n-i}\rho(\beta^{2}(x_{1}),\cdots,\widehat{\beta^{2}(x_{i})},\cdots,\beta^{2}(x_{n-1}),\alpha^{-1}\beta([\beta(y_{1}),\cdots,\alpha(y_{n})]_{\mathfrak{g}}))
f(\alpha\beta(x_i))\\
&-&\rho(\beta^{2}(x_{1}),\cdots,\beta^{2}(x_{n-1}))f([\beta(y_{1}),\cdots,\alpha(y_{n})]_{\mathfrak{g}})\\
&+&\rho(\beta^{2}(x_{1}),\cdots,\beta^{2}(x_{n-1}))\Big(f([\beta(y_1),\cdots,\beta(y_{n-1}),\alpha(y_n)]_\mathfrak{g})\\
&-&\displaystyle{\sum_{i=1}^{n-1}}(-1)^{n-i}\rho(\beta(y_1),\cdots,\widehat{\beta(y_i)},\cdots,\beta(y_{n-1}),\beta(y_n))
f(\alpha(y_i))-\rho(\beta(y_1),\cdots,\beta(y_{n-1}))f(\alpha(y_n))\Big)\\
&=&f([\beta^{2}(x_{1}),\cdots,\beta^{2}(x_{n-1}),[\beta(y_{1}),\cdots,\beta(y_{n-1}),\alpha(y_{n})]_{\mathfrak{g}}]_\mathfrak{g})\\
&-&\displaystyle{\sum_{i=1}^{n-1}}(-1)^{n-i}\rho(\beta^{2}(x_{1}),\cdots,\widehat{\beta^{2}(x_{i})},\cdots,\beta^{2}(x_{n-1}),\alpha^{-1}\beta([\beta(y_{1}),\cdots,\alpha(y_{n})]_{\mathfrak{g}}))
f(\alpha\beta(x_i))\\
&+&\rho(\beta^{2}(x_{1}),\cdots,\beta^{2}(x_{n-1}))\Big(
-\displaystyle{\sum_{i=1}^{n-1}}(-1)^{n-i}\rho(\beta(y_1),\cdots,\widehat{\beta(y_i)},\cdots,\beta(y_{n-1}),\beta(y_n))
f(\alpha(y_i))\\
&-&\rho(\beta(y_1),\cdots,\beta(y_{n-1}))f(\alpha(y_n))\Big)\\
&=&\displaystyle{\sum_{i=1}^{n}} (-1)^{n-i}f\Big([\beta^{2}(y_{1}),\ldots,\beta^{2}(y_{i-1}),\beta^{2}(y_{i+1}),\ldots,\beta^{2}(y_{n}),[\beta(x_{1}),\ldots, \beta(x_{n-1}),\alpha(y_{i})]_{\mathfrak{g}}]_\mathfrak{g}\Big)\\
&+&\displaystyle{\sum_{i=1}^{n-1}}(-1)^{n-i}\Big\{\displaystyle{\sum_{k=1}^{n-1}}(-1)^{n-k}\Big(\rho(\beta^{2}(y_{1}),\cdots,\widehat{\beta^{2}(y_{k})},\cdots,\beta^{2}(y_{n-1}),\beta^{2}(y_{n}))\circ
\rho(\beta(x_{1}),\cdots,\\
&&\widehat{\beta(x_{i})},\cdots,\beta(x_{n-1}),\beta(y_k))\Big)+\rho(\beta^{2}(y_{1}),\cdots,\beta^{2}(y_{n-1}))\circ\rho(\beta(y_n),\beta(x_{1}),\cdots,
\widehat{\beta(x_{i})},\cdots,\\&&\beta(x_{n-1}))\Big\}f(\alpha(x_{i}))+\\
&+&\displaystyle{\sum_{i=1}^{n}}\Big\{\rho(\beta^{2}(y_{1}),\cdots,\widehat{\beta^{2}(y_{i})},\cdots,\beta^{2}(y_{n-1},\beta^{2}(y_{n})))\circ\rho(\beta(x_{1}),\cdots,\beta(x_{n-1}))f(\alpha(y_{i})\\&-&
\displaystyle{\sum_{k=1}^{n}}\Big(\rho((\beta^{2}(y_{1}),\cdots,\alpha^{-1}\beta[\beta(x_{1}),\cdots,\beta(x_{n-1}),\alpha(y_{k})]_{\mathfrak{g}},\beta^{2}(y_{k}),\cdots,\beta^{2}(y_{n-1})))\Big)f(\alpha\beta(y_{i}))\\
&=&\displaystyle{\sum_{k=1}^{n}(-1)^{n-k}}\Big\{f([\beta^{2}(y_{1}),\cdots,\widehat{y_k},\cdots,\beta^{2}(y_{n}),[\beta(x_{1}),\cdots,\beta(x_{n-1}),\alpha(y_{k})]_{\mathfrak{g}}]_\mathfrak{g})\\
&-&\displaystyle{\sum_{i=1}^{n-1}}(-1)^{n-i}\rho(\beta^{2}(y_{1}),\cdots,\widehat{\beta^{2}(y_{i})},\cdots,\beta^{2}(y_{n-1}),\alpha^{-1}\beta([\beta(x_{1}),\cdots,\alpha(y_{k})]_{\mathfrak{g}}))
f(\alpha\beta(y_i))\\
&+&\rho(\beta^{2}(y_{1}),\cdots,\widehat{\beta^{2}(y_k)},\cdots,\beta^{2}(y_{n}))\Big(
-\displaystyle{\sum_{i=1}^{n-1}}(-1)^{n-i}\rho(\beta(x_1),\cdots,\widehat{\beta(x_i)},\cdots,\beta(x_{n-1}),\beta(y_k))
f(\alpha(x_i))\\
&-&\rho(\beta(x_1),\cdots,\beta(x_{n-1}))f(\alpha(y_k))\Big)\Big\}\\
&=& \displaystyle{\sum_{k=1}^{n}} (-1)^{n-k} \theta_{f}(\beta^{2}(y_{1}),\ldots,\beta^{2}(y_{k-1}),\beta^{2}(y_{k+1}),\ldots,\beta^{2}(y_{n}),[\beta(x_{1}),\ldots, \beta(x_{n-1}),\alpha(y_{k})]_{\mathfrak{g}})\\
&+& \displaystyle{\sum_{k=1}^{n}} (-1)^{n-k} \rho(\beta^{2}(y_{1}),\ldots,\beta^{2}(y_{k-1}),\beta^{2}(y_{k+1}),\ldots,\beta^{2}(y_{n}))\theta(\beta(x_{1}),\ldots, \beta(x_{n-1}),\alpha(y_{k}))\\

\end{array}$$

\end{proof}

\end{prop}
\begin{cor}
Under the above notations, $\theta+\theta_{f}$ is an $n$-cocycle associated with $\rho$.
\end{cor}
\begin{prop}
Under the above notations, $\sigma:T_{\theta}(\mathfrak{g})\rightarrow T_{\theta+\theta_{f}}(\mathfrak{g})$ is a isomorphism of $n$-BiHom- Lie algebras, where
$\sigma(v+x)=v+f(v)+x,~\forall v\in\mathfrak{g},~x\in V$.
\end{prop}
\begin{proof}
It is clear that $\sigma$ is a bijection.\\
Let $v_{i}\in\mathfrak{g},~x_{i}\in V,~i=1,2,\cdots,n$.
First, we have $\sigma\circ(\alpha+\alpha_{V})(v_{1}+x_{1})=\sigma(\alpha(v_{1})+\alpha_{V}(x_{1}))=\alpha(v_{1})
+f\circ\alpha(v_{1})+\alpha_{V}(x_{1})=\alpha(v_{1})+\alpha_{V}\circ f(v_{1})+\alpha_{V}(x_{1})=(\alpha+\alpha_{V})(v_{1}+f(v_{1})+x_{1})=(\alpha+\alpha_{V})\circ\sigma(v_{1}+x_{1}),$ then $\sigma\circ(\alpha+\alpha_{V})=(\alpha+\alpha_{V})\circ\sigma.$\\
Similarly, $\sigma\circ(\beta+\beta_{V})=(\beta+\beta_{V})\circ\sigma.$

Now, we have
$$\begin{array}{lllllll}&&[\sigma(v_{1}+x_{1}),\cdots,\sigma(v_{n}+x_{n})]_{\theta+\theta_{f}}\\
&=&[v_{1}+f(v_{1})+x_{1},\cdots,v_{n}+f(v_{n})+x_{n}]_{\theta+\theta_{f}}\\
&=&[v_{1},\cdots,v_{n}]_{\mathfrak{g}}+(\theta+\theta_{f})(v_{1},\cdots,v_{n})+\rho(v_{1},\cdots,v_{n-1})(f(v_{n})+x_{n})\\
&+&\displaystyle{\sum_{i=1}^{n-1}}(-1)^{n-i}\rho(v_1,\cdots,\widehat{v}_i,\cdots,v_{n-1},\alpha^{-1}\beta(v_n))
(\alpha_{V}\beta^{-1}_{V}(f(v_i)+x_i)))\\
&=&[v_{1},\cdots,v_{n}]_{\mathfrak{g}}+\theta(v_{1},\cdots,v_{n})
+f([v_{1},\cdots,v_{n}]_{\mathfrak{g}})\\
&-&\displaystyle{\sum_{i=1}^{n-1}}(-1)^{n-i}\rho(v_1,\cdots,\widehat{v}_i,\cdots,v_{n-1},\alpha^{-1}\beta(v_n))
f(\alpha_{V}\beta^{-1}_{V}(v_i))
-\rho(v_1,\cdots,v_{n-1})f(v_n)\\
&+&\displaystyle{\sum_{i=1}^{n-1}}(-1)^{n-i}\rho(v_1,\cdots,\widehat{v}_i,\cdots,v_{n-1},\alpha^{-1}\beta(v_n))
(f(\alpha_{V}\beta^{-1}_{V}(v_i)))+\rho(v_{1},\cdots,v_{n-1})(f(v_{n}))\\
&+&\displaystyle{\sum_{i=1}^{n-1}}(-1)^{n-i}\rho(v_1,\cdots,\widehat{v}_i,\cdots,v_{n-1},\alpha^{-1}_{V}\beta_{V}(v_n))
(\alpha_{V}\beta^{-1}_{V}(x_i))+\rho(v_{1},\cdots,v_{n-1})(x_{n})\\
&=&[v_{1},\cdots,v_{n}]_{\mathfrak{g}}+\theta(v_{1},\cdots,v_{n})+f([v_{1},\cdots,v_{n}]_{\mathfrak{g}})\\
&+&\displaystyle{\sum_{i=1}^{n-1}}(-1)^{n-i}\rho(v_1,\cdots,\widehat{v}_i,\cdots,v_{n-1},\alpha^{-1}\beta(v_n))
(\alpha_{V}\beta^{-1}_{V}(x_i))+\rho(v_{1},\cdots,v_{n-1})(x_{n})\\
&=&\sigma\Big( [v_{1},\cdots,v_{n}]_{\mathfrak{g}}+\theta(v_{1},\cdots,v_{n})+\displaystyle{\sum_{i=1}^{n-1}}(-1)^{n-i}\rho(v_1,\cdots,\widehat{v}_i,\cdots,v_{n-1},\alpha^{-1}\beta(v_n))
(\alpha_{V}\beta^{-1}_{V}(x_i))\\&+&\rho(v_{1},\cdots,v_{n-1})(x_{n}) \Big)\\
&=&\sigma([v_{1}+x_{1},\cdots,v_{n}+x_{n}]_{\theta}).
\end{array}$$

Then $\sigma:T_{\theta}(\mathfrak{g})\rightarrow T_{\theta+\theta_{f}}(\mathfrak{g})$ is a isomorphism of $n$-BiHom- Lie algebras.
\end{proof}
\subsection{$T_\theta^{\ast}$-extensions of $n$-BiHom-Lie algebras}
\begin{defn}
Let $\mathfrak{g}$ be an $n$-BiHom-Lie algebra over a field $\mathbb{K}$. We inductively define a derived series
$$(\mathfrak{g}^{(p)})_{p\geq 0}: \mathfrak{g}^{(0)}=\mathfrak{g},\ \mathfrak{g}^{(p+1)}=[\mathfrak{g}^{(p)},\cdots,\mathfrak{g}^{(p)},\mathfrak{g}]$$
and a central descending series
$$(\mathfrak{g}^{p})_{p\geq 0}: \mathfrak{g}^{0}=\mathfrak{g},\ \mathfrak{g}^{p+1}=[\mathfrak{g}^{p},\mathfrak{g},\cdots,\mathfrak{g}].$$

$\mathfrak{g}$ is called solvable and nilpotent $($of length $k$$)$ if and only if there is a $($smallest$)$ integer $k$ such that $\mathfrak{g}^{(k)}=0$ and $\mathfrak{g}^{k}=0$, respectively.
\end{defn}

\begin{thm}
Let $(\mathfrak{g},[\cdot,\cdots,\cdot],\alpha,\beta)$ be an $n$-BiHom-Lie algebra over a field $\mathbb{K}$.
\begin{enumerate}
   \item  If $\mathfrak{g}$ is solvable, then $(\mathfrak{g}\oplus \mathfrak{g}^{*},[\cdot,\cdots,\cdot]_\theta,\alpha+\tilde{\alpha},\beta+\tilde{\beta})$ is solvable.
   \item  If $\mathfrak{g}$ is nilpotent, then $(\mathfrak{g}\oplus \mathfrak{g}^{*},[\cdot,\cdots,\cdot]_\theta,\alpha+\tilde{\alpha},\beta+\tilde{\beta})$ is nilpotent.
\end{enumerate}
\end{thm}
\begin{proof}\begin{enumerate}
\item
We suppose that $\mathfrak{g}$ is solvable of length $s$, i.e. $\mathfrak{g}^{(s)}=[\mathfrak{g}^{(s-1)},\cdots,\mathfrak{g}^{(s-1)},\mathfrak{g}]=0.$ We claim that $(\mathfrak{g}\oplus \mathfrak{g}^*)^{(k)}\subseteq \mathfrak{g}^{(k)}+\mathfrak{g}^*$, which we prove by induction on $k$. The case $k=1$, by Proposition \ref{prop1}, we have
\begin{eqnarray*}
(\mathfrak{g}\oplus \mathfrak{g}^*)^{(1)}&=&[\mathfrak{g}\oplus \mathfrak{g}^*,\cdots,\mathfrak{g}\oplus \mathfrak{g}^*]_\theta\\
&=&[\mathfrak{g},\cdots,\mathfrak{g}]_\theta+\displaystyle{\sum_{i=1}^{n}}[\mathfrak{g},\cdots,\displaystyle{\underbrace{\mathfrak{g}^*}_{i}},\mathfrak{g},\cdots,\mathfrak{g}]_\theta\\
&=&[\mathfrak{g},\cdots,\mathfrak{g}]_{\mathfrak{g}}+\theta(\mathfrak{g},\cdots,\mathfrak{g})+\displaystyle{\sum_{i=1}^{n}}[\mathfrak{g},\cdots,\displaystyle{\underbrace{\mathfrak{g}^*}_{i}}
,\mathfrak{g},\cdots,\mathfrak{g}]_\theta\\
&\subseteq&\mathfrak{g}^{(1)}+\mathfrak{g}^*.
\end{eqnarray*}
By induction, $(\mathfrak{g}\oplus \mathfrak{g}^*)^{(k-1)}\subseteq \mathfrak{g}^{(k-1)}+\mathfrak{g}^*$. So
\begin{eqnarray*}
&&(\mathfrak{g}\oplus \mathfrak{g}^*)^{(k)}\\
&=&[(\mathfrak{g}\oplus \mathfrak{g}^*)^{(k-1)},\cdots,(\mathfrak{g}\oplus \mathfrak{g}^*)^{(k-1)},\mathfrak{g}\oplus \mathfrak{g}^*]_\theta\\
&\subseteq&[\mathfrak{g}^{(k-1)}+\mathfrak{g}^*,\cdots,\mathfrak{g}^{(k-1)}+\mathfrak{g}^*,\mathfrak{g}\oplus \mathfrak{g}^*]_\theta\\
&=&[\mathfrak{g}^{(k-1)},\cdots,\mathfrak{g}^{(k-1)},\mathfrak{g}]+\theta(\mathfrak{g}^{(k-1)},\cdots,\mathfrak{g}^{(k-1)},\mathfrak{g})+[\mathfrak{g}^{(k-1)},\cdots,\mathfrak{g}^{(k-1)},
\mathfrak{g}^*]_\theta\\&+&\displaystyle{\sum_{i=1}^{n-1}}[\mathfrak{g}^{(k-1)},\cdots,\displaystyle{\mathfrak{g}^{(k-1)}},\displaystyle{\underbrace{\mathfrak{g}^*}
_{i}},\mathfrak{g}^{(k-1)},\cdots,\mathfrak{g}^{(k-1)},\mathfrak{g}]_\theta\\
&\subseteq&\mathfrak{g}^{(k)}+\mathfrak{g}^*.
\end{eqnarray*}
Therefore
\begin{eqnarray*}
&&(\mathfrak{g}\oplus \mathfrak{g}^*)^{(s+1)}\\
&\subseteq&[\mathfrak{g}^{(s)},\cdots,\mathfrak{g}^{(s)},\mathfrak{g}]+\theta(\mathfrak{g}^{(s)},\cdots,\mathfrak{g}^{(s)},\mathfrak{g})+[\mathfrak{g}^{(s)},\cdots,\mathfrak{g}^{(s)},\mathfrak{g}^*]_\theta\\
&+&\displaystyle{\sum_{i=1}^{n-1}}[\mathfrak{g}^{(s)},\cdots,\displaystyle{\mathfrak{g}^{(s)}},\displaystyle{\underbrace{\mathfrak{g}^*}_{i}},\mathfrak{g}
^{(s)},\cdots,\mathfrak{g}^{(s)},\mathfrak{g}]_\theta\\
&=&0.
\end{eqnarray*}
It follows $(\mathfrak{g}\oplus \mathfrak{g}^{*},[\cdot,\cdots,\cdot]_\theta,\alpha+\tilde{\alpha},\beta+\tilde{\beta})$ is solvable.

\item Suppose that $\mathfrak{g}$ is nilpotent of length $s$. Since $(\mathfrak{g}\oplus \mathfrak{g}^{*})^{s}/\mathfrak{g}^{*}\cong \mathfrak{g}^{s}$ and $\mathfrak{g}^{s}=0$, we have
$(\mathfrak{g}\oplus \mathfrak{g}^{*})^{s}\subseteq \mathfrak{g}^{*}$. Let $h\in(\mathfrak{g}\oplus \mathfrak{g}^{*})^{s}\subseteq \mathfrak{g}^{*},~ b\in \mathfrak{g},~x_{i}^{j}+f_{i}^{j}\in \mathfrak{g}\oplus \mathfrak{g}^{*}, ~1\leq i\leq s-1,~~1\leq j\leq n-1$, we have
\begin{eqnarray*}
&&[[\cdots[h,x_1^{1}+f_{1}^{1},\cdots, x_{1}^{n-1}+f_{1}^{n-1}]_\theta,\cdots]_\theta,x_{s-1}^{1}+f_{s-1}^{1},\cdots,x_{s-1}^{n-1}+f_{s-1}^{n-1}]_\theta(b)\\
&=&(-1)^{s-1}h\alpha\beta^{-1}\ad(x_1^{1},\cdots,\alpha^{-1}\beta(x_1^{2}),\alpha^{-1}\beta(x_1^{n-1})\!)\alpha\beta^{-1}\ad(x_2^{1},\alpha^{-1}\beta(x_2^{2})\cdots,\\&&
\alpha^{-1}\beta(x_2^{n-1})\!)\!\cdots\!\alpha\beta^{-1}\ad(x_{s-1}^{1},\!\cdots,\alpha^{-1}\beta(x_{s-1}^{2}),\alpha^{-1}\beta(x_{s-1}^{n-1})\!)\!(b)\\
&=&(-1)^{s-1}h\alpha\beta^{-1}[x_1^{1},\cdots,\alpha^{-1}\beta(x_1^{2}),\alpha^{-1}\beta(x_1^{n-1}),\alpha\beta^{-1}[x_2^{1},\alpha^{-1}\beta(x_2^{2})\cdots,\\
&&\alpha^{-1}\beta(x_2^{n-1})),\cdots,\alpha\beta^{-1}[x_{s-1}^{1},\cdots,\alpha^{-1}\beta(x_{s-1}^{2}),\alpha^{-1}\beta(x_{s-1}^{n-1}),b]_{\mathfrak{g}}\cdots]_{\mathfrak{g}}\\
&\in &h(\mathfrak{g}^{s})=0.
\end{eqnarray*}
Thus $(\mathfrak{g}\oplus \mathfrak{g}^{*},[\cdot,\cdots,\cdot]_\theta,\alpha+\tilde{\alpha},\beta+\tilde{\beta})$ is nilpotent.
\end{enumerate}
\end{proof}

Now we consider the following symmetric bilinear form $q_{\mathfrak{g}}$ on $\mathfrak{g}\oplus \mathfrak{g}^{*}$,
$$q_{\mathfrak{g}}(x+f,y+g)=f(y)+g(x),~ \forall\, x+f, y+g\in \mathfrak{g}\oplus \mathfrak{g}^*.$$
Obviously, $q_{\mathfrak{g}}$ is  nondegenerate. In fact, if $x+f$ is orthogonal to  all elements $y+g$ of $\mathfrak{g}\oplus \mathfrak{g}^{*}$, then $f(y)=0$ and $g(x)=0$,  which implies that $x=0$ and $f=0$.

\begin{lem}\label{lemma3.2}
Let $q_\mathfrak{g}$ be as above. Then the 4-tuple $(\mathfrak{g}\oplus \mathfrak{g}^{*},q_\mathfrak{g},\alpha+\tilde{\alpha},\beta+\tilde{\beta})$ is a quadratic $n$-BiHom-Lie algebra if and only if $\theta$ satisfies for all $ x_1,\cdots,x_{n+1}\in \mathfrak{g}$,
\begin{eqnarray*}\label{203}
\theta(\beta(x_1), \cdots, \alpha(x_{n}))(\alpha(x_{n+1}))+\theta(\beta(x_1),\cdots,\alpha(x_{n+1}))(\alpha(x_{n}))=0.
\end{eqnarray*}
\end{lem}
\begin{proof}

Now suppose that  $x_i+f_i\in \mathfrak{g}\oplus \mathfrak{g}^*, i=1,\cdots,n+1$, we have
\begin{eqnarray*}
q_{\mathfrak{g}}((\alpha+\tilde{\alpha})(x_1+f_1), x_2+f_2)&=&q_\mathfrak{g}(\alpha(x_1)+f_1\circ\alpha,x_2+f_2)\\
&=&f_2\circ\alpha(x_1)+f_1(\alpha(x_2))\\
&=&q_\mathfrak{g}(x_1+f_1,(\alpha+\tilde{\alpha})(x_2+f_2)).
\end{eqnarray*}
Then $\alpha+\tilde{\alpha}$ is $q_\mathfrak{g}$-symmetric. 

In the same way, $\beta+\tilde{\beta}$ is $q_\mathfrak{g}$-symmetric.

Next, we can obtain
\begin{eqnarray*}
&&q_{\mathfrak{g}}\big([(\beta+\tilde{\beta})(x_1+f_1),\cdots,(\beta+\tilde{\beta})(x_{n-1}+f_{n-1}),(\alpha+\tilde{\alpha})(x_{n}+f_{n})]_\theta, (\alpha+\tilde{\alpha})(x_{n+1}+f_{n+1})\big)\\
&&+q_{\mathfrak{g}}\big((\alpha+\tilde{\alpha})(x_{n-1}+f_{n-1}),[(\beta+\tilde{\beta})(x_1+f_1),\cdots,(\beta+\tilde{\beta})(x_{n-1}+f_{n-1}) (\alpha+\tilde{\alpha})(x_{n+1}+f_{n+1})]_\theta\big)\\
&=&q_{\mathfrak{g}}\big([\beta(x_1)+f_1\circ\beta, \cdots,\beta(x_{n})+f_{n}\circ\beta]_\theta, \alpha(x_{n+1})+f_{n+1}\circ\alpha\big)\\
&&+q_{\mathfrak{g}}\big(\alpha(x_n)+f_n\circ\alpha, [\beta(x_1)+f_1\circ\beta,\cdots, \alpha(x_{n+1})+f_{n+1}\circ\alpha]_\theta\big)\\
&=&q_{\mathfrak{g}}\big([\beta(x_1),\cdots,\alpha(x_n)]_\mathfrak{g}+\theta(\beta(x_1),\cdots,\alpha(x_n))\\&+&\ad^*(\beta(x_1),\cdots,\beta(x_{n-1}))(f_n\circ\alpha)\\&+&
\displaystyle{\sum_{i=1}^{n-1}}(-1)^{n-i}\ad^{*}(\beta(x_1),\cdots,\widehat{\beta(x_i)},\cdots,\beta(x_{n-1}),\alpha^{-1}\beta(\alpha(x_n)))
(\tilde{\alpha}\tilde{\beta}^{-1}(f_{i}\circ\beta),\\
&&\alpha(x_{n+1})+f_{n+1}\circ\alpha\big)
+q_{\mathfrak{g}}\big(\alpha(x_n)+f_n\circ\alpha,[\beta(x_1),\cdots,\alpha(x_n)]_\mathfrak{g}+\theta(\beta(x_1),\cdots,\alpha(x_n))\\
&+&\ad^*(\beta(x_1),\cdots,\beta(x_{n-1}))(f_{n+1}\circ\alpha)\\&+&\displaystyle{\sum_{i=1}^{n-1}}(-1)^{n-i}\ad^{*}(\beta(x_1),\cdots,\widehat{\beta(x_i)},\cdots,\beta(x_{n-1}),\alpha^{-1}\beta(\alpha(x_{n+1})))
(\tilde{\alpha}\tilde{\beta}^{-1}(f_{i}\circ\beta)\\
&=&\theta(\beta(x_1),\cdots,\alpha(x_n))(\alpha(x_{n+1}))\\&-&
f_n\circ\alpha[\beta(x_1),\cdots,\beta(x_{n-1}),\alpha(x_{n+1})]_\mathfrak{g}\\&-&
\displaystyle{\sum_{i=1}^{n-1}}(-1)^{n-i}f_i\alpha[\beta(x_1),\cdots,\widehat{\beta(x_i)},\cdots,\beta(x_{n-1}),\beta(x_n),\alpha(x_{n+1})]_\mathfrak{g}\\&+&f_{n+1}\circ
\alpha[\beta(x_1),\cdots,\beta(x_{n-1}),\alpha(x_n)]_\mathfrak{g}
\\
&+&\theta(\beta(x_1),\cdots,\beta(x_{n-1}),\alpha(x_{n+1}))(\alpha(x_{n}))\\&-&
f_{n+1}\circ\alpha[\beta(x_1),\cdots,\beta(x_{n-1}),\alpha(x_{n})]_\mathfrak{g}\\&-&
\displaystyle{\sum_{i=1}^{n-1}}(-1)^{n-i}f_i\alpha[\beta(x_1),\cdots,\widehat{\beta(x_i)},\cdots,\beta(x_{n-1}),\beta(x_{n+1}),\alpha(x_{n})]_\mathfrak{g}\\&+&f_n\circ\alpha[\beta(x_1),\cdots,\beta(x_{n-1}),\alpha(x_{n+1})]_\mathfrak{g}\\
&=&\theta(\beta(x_1),\cdots,\alpha(x_n))(\alpha(x_{n+1}))+
\theta(\beta(x_1),\cdots,\beta(x_{n-1}),\alpha(x_{n+1}))(\alpha(x_{n}))
\end{eqnarray*}
which implies
\begin{eqnarray*}
&&q_{\mathfrak{g}}\big([(\beta+\tilde{\beta})(x_1+f_1),\cdots,(\beta+\tilde{\beta})(x_{n-1}+f_{n-1}),(\alpha+\tilde{\alpha})(x_{n}+f_{n})]_\theta, (\alpha+\tilde{\alpha})(x_{n+1}+f_{n+1})\big)\\
&&+q_{\mathfrak{g}}\big((\alpha+\tilde{\alpha})(x_{n-1}+f_{n-1}),[(\beta+\tilde{\beta})(x_1+f_1),\cdots,(\beta+\tilde{\beta})(x_{n-1}+f_{n-1}) (\alpha+\tilde{\alpha})(x_{n+1}+f_{n+1})]_\theta\big)=0
\end{eqnarray*}
if and only if
$\theta(\beta(x_1),\cdots,\alpha(x_n))(\alpha(x_{n+1}))+
\theta(\beta(x_1),\cdots,\beta(x_{n-1}),\alpha(x_{n+1}))(\alpha(x_{n}))
=0$.

Hence the lemma follows.
\end{proof}

Now, we shall call the quadratic $n$-BiHom-Lie algebra $(\mathfrak{g}\oplus \mathfrak{g}^{*},q_\mathfrak{g},\alpha+\tilde{\alpha},\beta+\tilde{\beta})$ the $T^*_\theta$-extension of $\mathfrak{g}$ (by $\theta$) and denote by $T_\theta^*(\mathfrak{g})$.

\begin{lem}\label{lemma3.1}
Let $(\mathfrak{g},q_\mathfrak{g},\alpha,\beta)$ be a $2m$-dimensional quadratic $n$-BiHom-Lie algebra over a field $\mathbb{K}$ $(ch\mathbb{K}\neq2)$, $\alpha$ be surjective and $I$ be an isotropic $m$-dimensional subspace of $\mathfrak{g}$. If $I$ is a BiHom-ideal of $(\mathfrak{g},[\cdot,\cdots,\cdot]_{\mathfrak{g}},\alpha,\beta)$, then $$[\beta(I),\beta(\mathfrak{g}),\cdots,\beta(\mathfrak{g}),\alpha(I)]_{\mathfrak{g}}=0$$.
\end{lem}
\begin{proof}
Since dim$I$+dim$I^{\bot}=m+\dim I^{\bot}=2m$ and $I\subseteq I^{\bot}$, we have $I=I^{\bot}$.
If $I$ is a BiHom-ideal of $(\mathfrak{g},[\cdot,\cdots,\cdot],\alpha,\beta)$, then \begin{eqnarray*}
q_\mathfrak{g}([\beta(I),\beta(\mathfrak{g}),\cdots,\beta(\mathfrak{g}),\alpha(I^{\bot})]_{\mathfrak{g}},\alpha(\mathfrak{g}))&=&
-q_\mathfrak{g}(\alpha(I^{\bot}),[\beta(I),\beta(\mathfrak{g}),\cdots,\alpha(\mathfrak{g})]_{\mathfrak{g}})\\
&\subseteq&q_\mathfrak{g}(\alpha(I^{\bot}),[I,\beta(\mathfrak{g}),\cdots,\alpha(\mathfrak{g})]_{\mathfrak{g}})\\
&\subseteq&q_\mathfrak{g}(I^{\bot},I)=0,
\end{eqnarray*}
which implies $[\beta(I),\beta(\mathfrak{g}),\cdots,\beta(\mathfrak{g}),\alpha(I)]_{\mathfrak{g}}=[\beta(I),\beta(\mathfrak{g}),\cdots,\beta(\mathfrak{g}),\alpha(I^{\bot})]_{\mathfrak{g}}\subseteq \alpha(\mathfrak{g})^{\bot}=\mathfrak{g}^{\bot}=0$.
\end{proof}

\begin{thm}
Let $(\mathfrak{g},q_\mathfrak{g},\alpha,\beta)$ be a quadratic regular $n$-BiHom-Lie algebra of dimensional $2m$ over a field $\mathbb{K}$ $(ch\mathbb{K}\neq2)$. Then $(\mathfrak{g},q_\mathfrak{g},\alpha,\beta)$ is isometric to a $T^{*}_\theta$-extension $(T_{\theta}^{*}(B),q_{B},\alpha^{'},\beta^{'})$ if and only if $(\mathfrak{g},[\cdot,\cdots,\cdot]_{\mathfrak{g}},\alpha,\beta)$ contains an isotropic BiHom-ideal $I$ of dimension $m$. In particular, $B\cong \mathfrak{g}/I$.
\end{thm}
\begin{proof}
($\Longrightarrow$) Suppose $\phi:B\oplus B^*\rightarrow \mathfrak{g}$ is isometric, we have $\phi(B^*)$ is a $m$-dimensional isotropic BiHom-ideal of $\mathfrak{g}$. In fact, since $\phi$ is isometric, ${\rm dim}B\oplus B^*={\rm dim}\mathfrak{g}=2m$, which implies ${\rm dim}B^*={\rm dim}\phi(B^*)=m$. And $0=q_B(B^*,B^*)=q_\mathfrak{g}(\phi(B^*),\phi(B^*))$, we have $\phi(B^*)\subseteq\phi(B^*)^{\bot}$. By $[\phi(B^*),\mathfrak{g},\cdots,\mathfrak{g}]=[\phi(B^*),\phi(B\oplus B^*),\cdots,\phi(B\oplus B^*)]=\phi([B^*,B\oplus B^*,\cdots,B\oplus B^*]_\theta)\subseteq\phi(B^*)$, then $\phi(B^*)$ is a BiHom-ideal of $\mathfrak{g}$. Furthermore, $B\cong B\oplus B^*/B^*\cong \mathfrak{g}/\phi(B^*)$.

($\Longleftarrow$) Suppose that $I$ is a $m$-dimensional isotropic BiHom-ideal of $\mathfrak{g}$. By Lemma \ref{lemma3.1},  $[\beta(I),\beta(\mathfrak{g}),\cdots,\beta(\mathfrak{g}),\alpha(I)]=0$. Let $B=\mathfrak{g}/I$ and $p: \mathfrak{g} \rightarrow B$ be the canonical projection. We can choose an isotropic complement subspace $B_{0}$ to $I$ in $\mathfrak{g}$, i.e. $\mathfrak{g}=B_{0}\dotplus I$ and $B_{0}\subseteq B_{0}^{\bot}$. Then $B_{0}^{\bot}=B_{0}$ since dim$B_0=m$.

Denote by $p_{0}$ (resp. $p_{1}$) the projection $\mathfrak{g}=B_{0}\dotplus I \rightarrow B_{0}$ (resp. $\mathfrak{g}=B_{0}\dotplus I\rightarrow I$) and let $q_{\mathfrak{g}}^{*}:I \rightarrow B^{*}$ is a linear map, where $q_{\mathfrak{g}}^{*}(i)(\bar{x}):= q_{\mathfrak{g}}(i,x),~\forall\, i\in I, \bar{x}\in B=\mathfrak{g}/I$.
 We claim that $q_{\mathfrak{g}}^{*}$ is a vector space isomorphism. In fact, if $\bar{x}=\bar{y}$, then $x-y\in I$, hence $q_{\mathfrak{g}}(i,x-y)\in q_{\mathfrak{g}}(I,I)=0$ and
 so $q_{\mathfrak{g}}(i,x)=q_{\mathfrak{g}}(i,y)$, which implies $q_{\mathfrak{g}}^{*}$ is well-defined and it is easy to see that $q_{\mathfrak{g}}^{*}$ is linear. If
 $q_{\mathfrak{g}}^{*}(i)=q_{\mathfrak{g}}^{*}(j)$, then $q_{\mathfrak{g}}^{*}(i)(\bar{x})=q_{\mathfrak{g}}^{*}(j)(\bar{x}), ~\forall\, x\in \mathfrak{g}$, i.e. $q_{\mathfrak{g}}(i,x)=q_{\mathfrak{g}}(j,x)$,
 which implies $i-j\in \mathfrak{g}^\bot=0$, hence $q_{\mathfrak{g}}^{*}$ is injective. Note that $\dim I=\dim B^*=m$, then $q_{\mathfrak{g}}^{*}$ is surjective.

In addition, $q_{\mathfrak{g}}^{*}$ has the following property, $\forall\, x_{1},\cdots,x_{n}\in \mathfrak{g}, i\in I$,
\begin{eqnarray*}
q_{\mathfrak{g}}^{*}([\beta(x_{1}),\cdots,\beta(x_{n-1}),\alpha(i)]_{\mathfrak{g}})(\bar{\alpha}(\bar{x_{n}}))
&=&q_{\mathfrak{g}}([\beta(x_{1}),\cdots,\beta(x_{n-1}),\alpha(i)]_{\mathfrak{g}},\alpha(x_{n+1}))\\
&=&-q_{\mathfrak{g}}(\alpha(i),[\beta(x_{1}),\cdots,\beta(x_{n-1}),\alpha(x_{n})])\\
&=&-q_{\mathfrak{g}}^{*}(\alpha(i))(\overline{[\beta(x_1),\cdots,\beta(x_{n-1}),\alpha(x_{n})]_{\mathfrak{g}}})\\
&=&-q_{\mathfrak{g}}^{*}(\alpha(i))([\overline{\beta(x_1)},\cdots,\overline{\beta(x_{n-1})},\overline{\alpha(x_{n})}]_{\mathfrak{g}})\\
&=&-q_{\mathfrak{g}}^{*}(\alpha(i))\ad(\overline{\beta(x_1)},\cdots,\overline{\beta(x_{n-1})})(\overline{\alpha(x_n)})\\
&=&\ad^*(\overline{\beta(x_1)},\cdots,\overline{\beta(x_{n-1})})q_{\mathfrak{g}}^{*}(\alpha(i))(\overline{\alpha(x_n)}).
\end{eqnarray*}
A similar computation shows that
$$q_{\mathfrak{g}}^{*}([\beta(x_{1}),\cdots,\beta(x_{k-1}),\beta(i),\beta(x_{k}),\cdots,\alpha(x_{n-1})]_{\mathfrak{g}})=-\ad^*(\overline{\beta(x_1)},\cdots,\overline{\beta(x_{n-1})})
q_{\mathfrak{g}}^{*}(\alpha(i)),$$
Define an $n$-linear map
\begin{eqnarray*}
\theta:~~~B\times \cdots\times B&\longrightarrow&B^{*}\\
(\bar{b_1},\cdots,\bar{b_n})&\longmapsto&q_{\mathfrak{g}}^{*}(p_{1}([b_1,\cdots,b_n]_{\mathfrak{g}})),
\end{eqnarray*}
where $b_1,\cdots,b_n\in B_{0}.$ Then $\theta$ is well-defined since $p|_{B_0}$ is a vector space isomorphism.

Now define the bracket $[\cdot,\cdots,\cdot]_\theta$ on $B\oplus B^*$ by Proposition \ref{prop1}, we have $B\oplus B^*$
is a algebra. Let $\varphi:\mathfrak{g} \rightarrow B\oplus B^{*}$ be a linear map defined by $\varphi(x+i)=\bar{x}+q_{\mathfrak{g}}^{*}(i),~ \forall\, x+i\in B_0\dotplus I=\mathfrak{g}. $ Since $p|_{B_0}$ and $q_{\mathfrak{g}}^{*}$ are vector space isomorphisms, $\varphi$ is also a vector space isomorphism. Note that $\varphi\alpha(x+i)=\varphi(\alpha(x)+\alpha(i))=\overline{\alpha(x)}+q_{\mathfrak{g}}^{*}(\alpha(i))=\overline{\alpha(x)}+q_{\mathfrak{g}}^{*}(i)\bar{\alpha}
 =(\bar{\alpha}+\tilde{\bar{\alpha}})(\bar{x}+q_{\mathfrak{g}}^{*}(i))=(\bar{\alpha}+\tilde{\bar{\alpha}})\varphi(x+i)$, i.e. $\varphi\alpha=(\bar{\alpha}+\tilde{\bar{\alpha}})\varphi$. By the same way, $\varphi\beta=(\bar{\beta}+\tilde{\bar{\beta}})\varphi$.

Furthermore, $\forall\, x_i\in \mathfrak{g}$, $z_i\in I$,
$$\begin{array}{lllllll}
&&\varphi([\beta(x_1+z_1),\cdots,\beta(x_{n-1}+z_{n-1}),\alpha(x_n+z_n)]_{\mathfrak{g}})=\varphi\Big([\beta(x_1),\cdots,\beta(x_{n-1}),\alpha(x_n)]_{\mathfrak{g}}\\&&
+\sum_{k=1}^n[\beta(x_1),\cdots,\beta(z_k),\cdots,\alpha(x_n)]_{\mathfrak{g}}\Big)\\
&=&\varphi\Big(p_{0}([\beta(x_1),\cdots,\beta(x_{n-1}),\alpha(x_n)]_{\mathfrak{g}})+p_1([\beta(x_1),\cdots,\beta(x_{n-1}),\alpha(x_n)]_{\mathfrak{g}})+\sum_{k=1}^n[\beta(x_1),
\cdots,\\&&\beta(z_k),\cdots,\alpha(x_n)]_{\mathfrak{g}}\Big)\\
&=&\overline{[x_1,\cdots,x_n]_{\mathfrak{g}}}+f^*_1\left(p_1([\beta(x_1),\cdots,\beta(x_{n-1}),\alpha(x_n)]_{\mathfrak{g}})
+\sum_{k=1}^n[\beta(x_1),\cdots,\beta(z_k),\cdots,\alpha(x_n)]_{\mathfrak{g}}\right)\\
&=&[\overline{\beta(x_1)},\cdots,\overline{\beta(x_{n-1})},\overline{\alpha(x_n)}]_{\mathfrak{g}_1}+\theta(\overline{\beta(x_1)},\cdots,,\beta(\overline{x_{n-1}}),\overline{\alpha(x_n)})\\
&+&\sum_{k=1}^n(-1)^{n-k}\ad^*(\overline{\beta(x_1)},\cdots,\widehat{\overline{\beta(x_k)}},\cdots,\overline{\beta(x_{n-1})},\overline{\beta(x_n)})q_{\mathfrak{g}}^*(\alpha(z_k))\\
&=&[\overline{\beta(x_1)}+q_{\mathfrak{g}}^*(\beta(z_1)),\cdots,\overline{\beta(x_{n-1})}+q_{\mathfrak{g}}^*(\beta(z_{n-1})),\overline{\alpha(x_n))}+q_{\mathfrak{g}}^*(\alpha(z_n))]_{\theta}\\
&=&[\varphi(\beta(x_1)+\beta(z_1)),\cdots,\varphi(\beta(x_{n-1})+\beta(z_{n-1})),\varphi(\alpha(x_n)+\alpha(z_n))]_{\theta},
\end{array}$$
Then $\varphi$ is an isomorphism of algebras, and $(B\oplus B^{*},[\cdot,\cdots,\cdot]_\theta, \bar{\alpha}+\tilde{\bar{\alpha}},\bar{\beta}+\tilde{\bar{\beta}})$ is an $n$-BiHom-Lie algebra.
Furthermore, we have
{\setlength\arraycolsep{2pt}
\begin{eqnarray*}
q_{B}(\varphi(x+i),\varphi(y+j))&=&q_{B}(\bar{x}+q_{\mathfrak{g}}^{*}(i),\bar{y}+q_{\mathfrak{g}}^{*}(j))\\
&=&q_{\mathfrak{g}}^{*}(i)(\bar{y})+q_{\mathfrak{g}}^{*}(j)(\bar{x})\\
&=&q_{\mathfrak{g}}(i,y)+q_{\mathfrak{g}}(j,x)\\
&=&q_{\mathfrak{g}}(x+i,y+j),
\end{eqnarray*}}
then $\varphi$ is isometric. And $\forall\, x_i\in \mathfrak{g}$, the relation
\begin{eqnarray*}
&&q_{B}([(\bar{\beta}+\tilde{\bar{\beta}})(\varphi(x_{1}),\cdots,(\bar{\beta}+\tilde{\bar{\beta}})(\varphi(x_{n-1})),(\bar{\alpha}+\tilde{\bar{\alpha}})(\varphi(x_{n}))]_\theta,(\bar{\alpha}+\tilde{\bar{\alpha}})(\varphi(x_{n+1})))\\
&=&q_{B}([\varphi(\beta(x_{1})),\cdots,\varphi(\beta(x_{n-1})),\varphi(\alpha(x_{n}))]_\theta,\varphi(\alpha(x_{n+1})))\\&=&q_B(\varphi([\beta(x_1),\cdots,\beta(x_{n-1}),\alpha(x_n)]),\varphi(\alpha(x_{n+1})))\\
&=&q_{\mathfrak{g}}([\beta(x_1),\cdots,\beta(x_{n-1}),\alpha(x_n)],\alpha(x_{n+1}))\\&=&-q_{\mathfrak{g}}(\alpha(x_{n}),[\beta(x_1),\cdots,\beta(x_{n-1}),\alpha(x_{n+1})])\\
&=&-q_{B}(\varphi(\alpha(x_n)),[\varphi(\beta(x_1)),\cdots,\varphi(\beta(x_{n-1})),\varphi(\alpha(x_{n+1}))]_\theta)\\
&=&-q_{B}((\bar{\beta}+\tilde{\bar{\beta}})(\varphi(x_n)),[(\bar{\beta}+\tilde{\bar{\beta}})(\varphi(x_1)),\cdots,(\bar{\beta}+\tilde{\bar{\beta}})(\varphi(x_{n-1})),(\bar{\alpha}+\tilde{\bar{\alpha}})(\varphi(x_{n+1}))]_\theta),
\end{eqnarray*}
 which implies that $q_B$ is $\alpha\beta$-invariant. So $(B\oplus B^{*}, q_B,\bar{\beta}+\tilde{\bar{\beta}},\bar{\alpha}+\tilde{\bar{\alpha}})$ is a quadratic $n$-BiHom-Lie algebra.
Thus, the $T_\theta^*$-extension $(B\oplus B^{*}, q_B,\bar{\beta}+\tilde{\bar{\beta}},\bar{\alpha}+\tilde{\bar{\alpha}})$ of $B$ is isometric to $(\mathfrak{g},q_{\mathfrak{g}},\alpha,\beta)$.
\end{proof}
\section{Deformations of $n$-BiHom-Lie algebras}
\begin{defn} Let $(\mathfrak{g}, [\cdot,\cdots,\cdot]_{\mathfrak{g}}, \alpha,\beta)$ be an $n$-BiHom-Lie algebra and $(V,\rho,\alpha_{V},\beta_{V})$ be a representation of $\mathfrak{g}$.
An $m$-cochain is an $(m+1)$-linear map $$f: \underbrace{\wedge^{n-1}\mathfrak{g}\otimes\cdots\otimes\wedge^{n-1}\mathfrak{g}}_m\wedge\mathfrak{g}\longrightarrow V$$
such that $$\begin{array}{llllll}\alpha_{V}\circ f(X_{1},X_{2},\cdots,X_{m},z)&=&f(\alpha(X_{1}),\alpha(X_{2}),\cdots,\alpha(X_{m}),\alpha(z))\\
\beta_{V}\circ f(X_{1},X_{2},\cdots,X_{m},z)&=&f(\beta(X_{1}),\beta(X_{2}),\cdots,\beta(X_{m}),\beta(z))
\end{array}$$ for all $X_{1},X_{2},\cdots,X_{m}\in \wedge^{n-1}\mathfrak{g}$ and $z\in \mathfrak{g}.$ We denote the set of $m$-cochain by $C^m(\mathfrak{g}, V).$
\end{defn}
\begin{defn}
For $m=1,2$, the coboundary operator $\delta^{m}: C^m(\mathfrak{g},\mathfrak{g}) \longrightarrow C^{m+1}(\mathfrak{g}, \mathfrak{g})$ is defined as follows:

For all $X=x_1\wedge\cdots\wedge x_{n-1},~Y=y_1\wedge\cdots\wedge y_{n-1}\in\wedge^{n-1}\mathfrak{g},~x_{n},y_{n}\in\mathfrak{g}$, we have
\begin{eqnarray}&&\delta^{1}f(X,x_{n})=-f([x_{1},\cdots,x_{n}]_{\mathfrak{g}})+\sum_{i=1}^{n}[x_{1},\cdots,f(x_{i}),\cdots,x_{n}]_{\mathfrak{g}}\label{isma}\\[0.2cm]
&&\delta^{2}f(X,Y,y_{n})=[\beta^{2}(x_{1}),\ldots,\beta^{2}(x_{n-1}),f(\beta(y_{1}),\ldots,\beta(y_{n-1}),\alpha(y_{n}))]_{\mathfrak{g}} \nonumber\\
&&+f(\beta^{2}(x_{1}),\ldots,\beta^{2}(x_{n-1}),[\beta(y_{1}),\ldots,\beta(y_{n-1}),\alpha(y_{n})]_{\mathfrak{g}})\nonumber\\
&&- \displaystyle{\sum_{i=1}^{n}}(-1)^{n-i}\Big([\beta^{2}(y_{1}),\ldots,\widehat{\beta^{2}(y_{i})},\ldots,\beta^{2}(y_{n}),f(\beta(x_{1}),\ldots, \beta(x_{n-1}),\alpha(y_{i}))]_{\mathfrak{g}}\nonumber\\&&+f(\beta^{2}(y_{1}),\ldots,\widehat{\beta^{2}(y_{i})},\ldots,\beta^{2}(y_{n}),[\beta(x_{1}),\ldots, \beta(x_{n-1}),\alpha(y_{i})]_{\mathfrak{g}})\Big)\label{6}.
\end{eqnarray}
\end{defn}
\begin{lem}
With respect to the above notations, for $f\in C^m(\mathfrak{g},\mathfrak{g})$ ( $m=1,2$), and by the multiplicative property, we have
$$\begin{array}{lllllll}\delta^{m} f\circ\alpha &=&\alpha\circ\delta^{m}f,\\
\delta^{m} f\circ\beta&=&\beta\circ\delta^{m}f,~~m=1,2.\end{array}$$
Thus the map $\delta^{m}: C^m(\mathfrak{g},\mathfrak{g}) \longrightarrow C^{m+1}(\mathfrak{g}, \mathfrak{g})$ is well defined.
\end{lem}
\begin{proof}
The proof is straighforward by a direct computation.
\end{proof}
\begin{thm}
The coboundary operator $\delta^{1}$ and $\delta^{2}$ defined above satisfy $\delta^{2}\circ\delta^{1}=0$.
\end{thm}
\begin{proof}
For any $f\in C^{1}(\mathfrak{g}, \mathfrak{g})$, we have
\begin{eqnarray}
&&\delta^{2}\circ\delta^{1}f(X,Y,y_{n})\nonumber\\
&=&[\beta^{2}(x_{1}),\ldots,\beta^{2}(x_{n-1}),\delta^{1}f(\beta(y_{1}),\ldots,\beta(y_{n-1}),\alpha(y_{n}))]_{\mathfrak{g}} \nonumber\\
&+&\delta^{1}f(\beta^{2}(x_{1}),\ldots,\beta^{2}(x_{n-1}),[\beta(y_{1}),\ldots,\beta(y_{n-1}),\alpha(y_{n})]_{\mathfrak{g}})\nonumber\\
&-& \displaystyle{\sum_{k=1}^{n}}(-1)^{n-k}\Big([\beta^{2}(y_{1}),\ldots,\widehat{\beta^{2}(y_{k})},\ldots,\beta^{2}(y_{n}),\delta^{1}f(\beta(x_{1}),\ldots, \beta(x_{n-1}),\alpha(y_{k}))]_{\mathfrak{g}}\nonumber\\&+&\delta^{1}f(\beta^{2}(y_{1}),\ldots,\widehat{\beta^{2}(y_{k})},\ldots,\beta^{2}(y_{n}),[\beta(x_{1}),\ldots, \beta(x_{n-1}),\alpha(y_{k})]_{\mathfrak{g}})\Big)\nonumber\\
&=&-[\beta^{2}(x_{1}),\cdots,\beta^{2}(x_{n-1}),f([\beta(y_{1}),\cdots,\beta(y_{n-1}),\alpha(y_{n})]_{\mathfrak{g}})]_{\mathfrak{g}}\label{1.0}\\
&+&[\beta^{2}(x_{1}),\cdots,\beta^{2}(x_{n-1}),\sum_{i=1}^{n}[\beta(y_{1}),\cdots,f(\beta(y_{i})),\cdots,\alpha(y_{n})]_{\mathfrak{g}}]_{\mathfrak{g}}\label{2.0}\\
&-&f([\beta^{2}(x_{1}),\cdots,\beta^{2}(x_{n-1}),[\beta(y_{1}),\cdots,\beta(y_{n-1}),\alpha(y_{n})]_{\mathfrak{g}}]_{\mathfrak{g}})\label{3.0}\\
&+&\sum_{i=1}^{n-1}[\beta^{2}(x_{1}),\cdots,f(\beta^{2}(x_{i})),\cdots,\beta^{2}(x_{n-1}),[\beta(y_{1}),\cdots,\beta(y_{n-1}),\alpha(y_{n})]_{\mathfrak{g}}]_{\mathfrak{g}}\label{4.0}\\
&+&[\beta^{2}(x_{1}),\cdots,\beta^{2}(x_{n-1}),f([\beta(y_{1}),\cdots,\beta(y_{n-1}),\alpha(y_{n})]_{\mathfrak{g}})]_{\mathfrak{g}}\label{5.0}\\
&+&\displaystyle{\sum_{k=1}^{n}}(-1)^{n-k}[\beta^{2}(y_{1}),\cdots,\widehat{\beta^{2}(y_{k})},\cdots,\beta^{2}(y_{n}),f([\beta(x_{1}),\cdots,\beta(x_{n-1}),\alpha(y_{k})]_{\mathfrak{g}})]_{\mathfrak{g}}\label{6.0}\\
&-&\displaystyle{\sum_{k=1}^{n}}(-1)^{n-k}[\beta^{2}(y_{1}),\cdots,\widehat{\beta^{2}(y_{k})},\cdots,\beta^{2}(y_{n}),\sum_{i=1}^{n-1}[\beta(x_{1}),\cdots,f(\beta(x_{i})),\cdots,
\beta(x_{n-1}),\alpha(y_{k})]_{\mathfrak{g}}]_{\mathfrak{g}}\label{7.0}\\
&-&\displaystyle{\sum_{k=1}^{n}}(-1)^{n-k}[\beta^{2}(y_{1}),\cdots,\widehat{\beta^{2}(y_{k})},\cdots,\beta^{2}(y_{n}),[\beta(x_{1}),\cdots,\beta(x_{n-1}),\alpha(y_{k})]_{\mathfrak{g}}]_{\mathfrak{g}}\label{8.0}\\
&+&\displaystyle{\sum_{k=1}^{n}}(-1)^{n-k}f([\beta^{2}(y_{1}),\cdots,\widehat{\beta^{2}(y_{k})},\cdots,\beta^{2}(y_{n}),[\beta(x_{1}),\cdots,\beta(x_{n-1}),\alpha(y_{k})]_{\mathfrak{g}}]_{\mathfrak{g}})\label{9.0}\\
&-&\displaystyle{\sum_{k=1}^{n}}\displaystyle{\sum_{i=1,i\neq k}^{n}}(-1)^{n-k}[\beta^{2}(y_{1}),\cdots,\widehat{\beta^{2}(y_{k})},\cdots,f(\beta^{2}(y_{i}),\cdots,\beta^{2}(y_{n}),[\beta(x_{1}),\cdots,\beta(x_{n-1}),\alpha(y_{k})]_{\mathfrak{g}}]_{\mathfrak{g}}\label{10.0}\\
&-&\displaystyle{\sum_{k=1}^{n}}(-1)^{n-k}[\beta^{2}(y_{1}),\cdots,\widehat{\beta^{2}(y_{k})},\cdots,\beta^{2}(y_{n}),f([\beta(x_{1}),\cdots,\beta(x_{n-1}),\alpha(y_{k})]_{\mathfrak{g}})]_{\mathfrak{g}}\label{11.0}.
\end{eqnarray}
By a direct computation, we get $(\ref{1.0})+(\ref{5.0})=(\ref{6.0})+(\ref{11.0})=0$, and by the $n$-Bihom-Jacobi condition, we obtain $(\ref{3.0})+(\ref{9.0})=(\ref{4.0})+(\ref{7.0})=(\ref{2.0})+(\ref{8.0})+(\ref{10.0})=0$.

Therefore $\delta^{2}\circ\delta^{1}=0,$ which completes the proof.
\end{proof}
For $m=1,2,$ the map $f\in C^m(\mathfrak{g},\mathfrak{g})$ is called an $m$-BiHom-cocycle $\delta^m f=0.$ We denote the subspace spanned by $m$-Bihom-cocycles by
 $Z^m(\mathfrak{g},\mathfrak{g})$ and $B^m(\mathfrak{g},\mathfrak{g})=\delta^{m-1}C^{m-1}(\mathfrak{g},\mathfrak{g}).$ Since  $\delta^2\circ\delta^1=0,$ $B^2(\mathfrak{g},\mathfrak{g})$ is a subspace of $Z^2(\mathfrak{g},\mathfrak{g}).$ Hence we can define a cohomology space $H^2(\mathfrak{g},\mathfrak{g})$ of as the factor space
 $Z^2(\mathfrak{g},\mathfrak{g})/B^2(\mathfrak{g},\mathfrak{g}).$
 \begin{defn}
Let $(\mathfrak{g},[\cdot,\cdots,\cdot]_{\mathfrak{g}},\alpha,\beta)$ be an $n$-BiHom-Lie algebras over $\mathbb{K}$. A deformation of $(\mathfrak{g},[\cdot,\cdots,\cdot]_{\mathfrak{g}},\alpha,\beta)$ is given by $\mathbb{K}[[t]]-n$-linear map
$$f_{t}=\sum_{p\geq0}f_{p}t^{p}:\mathfrak{g}[[t]]\times\cdots\times\mathfrak{g}[[t]]$$
such that $(\mathfrak{g}[[t]],f_t,\alpha,\beta)$ is also an $n$-BiHom-Lie algebras. We call $f_1$ the infinitesimal deformation of $(\mathfrak{g},[\cdot,\cdots,\cdot],\alpha,\beta)$.
 \end{defn}
Since $(\mathfrak{g}[[t]],f_t,\alpha,\beta)$ is an $n$-BiHom-Lie algebras, $f_t$ satisfies
\begin{equation}\label{1}
\alpha\circ f_t(x_{1},\cdots,x_n)=f_t(\alpha(x_1),\cdots,\alpha(x_n)),
\end{equation}
 \begin{equation}\label{2}
\beta\circ f_t(x_{1},\cdots,x_n)=f_t(\beta(x_1),\cdots,\beta(x_n)),
\end{equation}
\begin{equation}
\begin{array}{llll}\label{3}
&&f_t(\beta^{2}(x_{1}),\ldots,\beta^{2}(x_{n-1}),f_t(\beta(y_{1}),\ldots,\beta(y_{n-1}),\alpha(y_{n}))) \\&=& \sum_{k=1}^{n} (-1)^{n-k}f_t(\beta^{2}(y_{1}),\ldots,\beta^{2}(y_{k-1}),\beta^{2}(y_{k+1}),\ldots,\beta^{2}(y_{n}),f_t(\beta(x_{1}),\ldots, \beta(x_{n-1}),\alpha(y_{k})))
\end{array}
\end{equation}
(\ref{1}), (\ref{2}) and (\ref{3}) are equivalents to
\begin{equation}\label{4}
\alpha\circ f_p(x_{1},\cdots,x_n)=f_p(\alpha(x_1),\cdots,\alpha(x_n)),
\end{equation}
 \begin{equation}\label{5}
\beta\circ f_p(x_{1},\cdots,x_n)=f_p(\beta(x_1),\cdots,\beta(x_n)),
\end{equation}
\begin{equation}
\begin{array}{llll}\label{6}
&&\sum_{p+q=l}f_p(\beta^{2}(x_{1}),\ldots,\beta^{2}(x_{n-1}),f_q(\beta(y_{1}),\ldots,\beta(y_{n-1}),\alpha(y_{n}))) \\&=& \sum_{i=1}^{n}\Big(\sum_{p+q=l}(-1)^{n-i}f_p(\beta^{2}(y_{1}),\ldots,\beta^{2}(y_{k-1}),\beta^{2}(y_{k+1}),\ldots,\beta^{2}(y_{n}),f_q(\beta(x_{1}),\ldots, \beta(x_{n-1}),\alpha(y_{k}))\Big)
\end{array}
\end{equation}
We call these the deformation equations for an $n$-BiHom-Lie algebra.

(\ref{4}) and (\ref{5}) shows that $f_p\in C^{1}(\mathfrak{g},\mathfrak{g})$. In (\ref{6}), set $l=1$, then
$$\begin{array}{llll}\label{6.00}
&&[\beta^{2}(x_{1}),\ldots,\beta^{2}(x_{n-1}),f_1(\beta(y_{1}),\ldots,\beta(y_{n-1}),\alpha(y_{n}))]_{\mathfrak{g}} \\
&+&f_1(\beta^{2}(x_{1}),\ldots,\beta^{2}(x_{n-1}),[\beta(y_{1}),\ldots,\beta(y_{n-1}),\alpha(y_{n})]_{\mathfrak{g}})\\
&-& \displaystyle{\sum_{i=1}^{n}}(-1)^{n-i}\Big([\beta^{2}(y_{1}),\ldots,\beta^{2}(y_{i-1}),\beta^{2}(y_{i+1}),\ldots,\beta^{2}(y_{n}),f_1(\beta(x_{1}),\ldots, \beta(x_{n-1}),\alpha(y_{i}))]_{\mathfrak{g}}\\&+&f_1(\beta^{2}(y_{1}),\ldots,\beta^{2}(y_{i-1}),\beta^{2}(y_{i+1}),\ldots,\beta^{2}(y_{n}),[\beta(x_{1}),\ldots, \beta(x_{n-1}),\alpha(y_{i})]_{\mathfrak{g}})\Big)=0,
\end{array}$$
i.e., $\delta^{2}f_{1}(x_1,\cdots,x_{n-1},y_{1},\cdots,y_{n-1},y_n)=0$. Hence the infinitesimal deformation $f_1\in Z^{2}(\mathfrak{g},\mathfrak{g})$.
\begin{defn}
Two deformations $f_t$ and $f'_t$ of the $n$-BiHom-Lie algebras $(\mathfrak{g},[\cdot,\cdots,\cdot],\alpha,\beta)$ are said to be equivalent, if there exists a formal automorphism
of $n$-BiHom-Lie algebras $\Psi_{t}:(\mathfrak{g},f_t,\alpha,\beta)\longrightarrow(\mathfrak{g},f'_t,\alpha,\beta)$ that may be written in
the form $\Psi_{t}=\sum_{i\geq0}\psi_i t^{i},~\psi_{i}:\mathfrak{g}\longrightarrow\mathfrak{g}$ is a linear map such that
$$\psi_{0}=id_{\mathfrak{g}};~\psi_{i}\circ\alpha=\alpha\circ\psi_{i};~\psi_{i}\circ\beta=\beta\circ\psi_{i};$$
$$\Psi_{t}\circ f_{t}(x_1,\cdots,x_n)=f'_t(\Psi_{t}(x_1),\cdots,\Psi_{t}(x_{n}),$$
and is denoted by $f_t\equiv f'_t$. When $f_1=f_2=\cdots=0,~f_1=f_0$ is called the nulldeformation; if $f_t\equiv f_0$, then $f_t$ is called the trivial deformation.
\end{defn}
\begin{thm}
Let $f_t$ and $f'_{t}$ be two equivalent deformations of the $n$-BiHom-Lie algebra $(\mathfrak{g},[\cdot,\cdots,\cdot]_{\mathfrak{g}},\alpha,\beta)$. Then the infinitesimal deformations $f_1$ and $f'_1$ belong to the same cohomology class in the cohomology group $H^{2}(\mathfrak{g},\mathfrak{g})$
\end{thm}
\begin{proof}
Put $B^{2}(\mathfrak{g},\mathfrak{g}):=\delta^{1}C^{1}(\mathfrak{g},\mathfrak{g})$. It is enough to prove that $f_1-f'_1\in B^{2}(\mathfrak{g},\mathfrak{g})$. Let
$\Psi_{t}:(\mathfrak{g},f_t,\alpha,\beta)\longrightarrow(\mathfrak{g},f'_t,\alpha,\beta)$ be an isomorphism of $n$-BiHom-Lie algebras. Then $\psi_1\in C^{1}(\mathfrak{g},\mathfrak{g})$
and
$$\sum_{i\geq 0}\psi_i(\sum_{j\geq 0}f_j(x_{1},\cdots,x_{n}))t^{i+j}=\sum_{i\geq 0}f'_{i}(\sum_{j_{1}\geq 0}\psi_{j_{1}}(x_1),\cdots,\sum_{j_{n}\geq 0}\psi_{j_{n}}(x_{n}))t^{i+j_{1}+\cdots+j_{n}},$$
comparing with the coefficients of $t^{1}$ for two sides of the above equation, we obtain
$$\begin{array}{lllll}&&f_{1}(x_{1},\cdots,x_{n})+\psi_{1}([x_{1},\cdots,x_{n}]_{\mathfrak{g}})\\
&=&[\psi_{1}(x_{1},x_{2},\cdots,x_{n}]_{\mathfrak{g}}+[x_{1},\psi_{1}(x_{2}),\cdots,x_{n}]+\cdots+[x_{1},\cdots,x_{n-1},\psi_{1}(x_{n})]_{\mathfrak{g}}+f'_{1}(x_{1},\cdots,x_{n}).\end{array}$$
Furthermore, one gets
$$\begin{array}{lll}&&f_{1}(x_{1},\cdots,x_{n})-f'_{1}(x_{1},\cdots,x_{n})=-\psi_{1}([x_{1},\cdots,x_{n}]_{\mathfrak{g}})+
[\psi_{1}(x_{1},x_{2},\cdots,x_{n}]_{\mathfrak{g}}\\&&+[x_{1},\psi_{1}(x_{2}),\cdots,x_{n}]_{\mathfrak{g}}+\cdots+[x_{1},\cdots,x_{n-1},\psi_{1}(x_{n})]_{\mathfrak{g}}\\
&=&-\psi_{1}([x_{1},\cdots,x_{n}]_{\mathfrak{g}})+\sum_{i=1}^{n}[x_{1},\cdots,\psi(x_{i}),\cdots,x_{n}]_{\mathfrak{g}}\\
&=&\delta^{1}\psi_{1}(x_{1},\cdots,x_{n}).
\end{array}$$
Therefore, $f_{1}-f'_{1}=\delta^{1}\psi_{1}\in\delta^{1}C^{1}(\mathfrak{g},\mathfrak{g})$, i.e., $f_{1}-f'_{1}\in B^{2}(\mathfrak{g},\mathfrak{g}).$
\end{proof}
\begin{rmk}
An $n$-BiHom-Lie algebra $(\mathfrak{g},[\cdot,\cdots,\cdot]_{\mathfrak{g}},\alpha,\beta)$ is analytically
rigid if every deformation $f_t$ is equivalent to the null deformation $f_{0}$.
\end{rmk}
We have a fundamental theorem.
\begin{thm}
If $(\mathfrak{g},[\cdot,\cdots,\cdot]_{\mathfrak{g}},\alpha,\beta)$ is an $n$-BiHom-Lie algebra with $H^{2}(\mathfrak{g},\mathfrak{g})=0$, then $(\mathfrak{g},[\cdot,\cdots,\cdot]_{\mathfrak{g}},\alpha,\beta)$ is analytically rigid.
\end{thm}
\begin{proof}
Let $f_t$ be a deformation of the $n$-BiHom-Lie algebra $(\mathfrak{g},[\cdot,\cdots,\cdot]_{\mathfrak{g}},\alpha,\beta)$ with $f_t=f_0+f_{r}t^{r}+f_{r+1}t^{r+1}+\cdots,$ i.e., $f_{1}=f_{2}=\cdots=f_{r-1}=0$. Then set $l=r$ in (\ref{6}), we have
$$\begin{array}{lllll}&&f_r(\beta^{2}(x_{1}),\ldots,\beta^{2}(x_{n-1}),[\beta(y_{1}),\ldots,\beta(y_{n-1}),\alpha(y_{n})]_{\mathfrak{g}})\\
&+&[\beta^{2}(x_{1}),\ldots,\beta^{2}(x_{n-1}),f_r(\beta(y_{1}),\ldots,\beta(y_{n-1}),\alpha(y_{n}))]_{\mathfrak{g}} \\
&&-\sum_{i=1}^{n}(-1)^{n-i}\Big([\beta^{2}(y_{1}),\ldots,\beta^{2}(y_{k-1}),\beta^{2}(y_{k+1}),\ldots,\beta^{2}(y_{n}),f_r(\beta(x_{1}),\ldots,\\ &&\beta(x_{n-1}),\alpha(y_{k}))]_{\mathfrak{g}}+f_r(\beta^{2}(y_{1}),\ldots,\beta^{2}(y_{k-1}),\beta^{2}(y_{k+1}),\ldots,\beta^{2}(y_{n}),[\beta(x_{1}),\ldots, \\ &&\beta(x_{n-1}),\alpha(y_{k})]_{\mathfrak{g}})\Big)=0\end{array}$$
i.e, $\delta^{2}f_{r}(x_{1},\cdots,x_{n-1},y_{1},\cdots,y_{n})=0,$
$\delta^{2}(f_r)=0,$ that is, $f_r\in Z^2(\mathfrak{g}, \mathfrak{g}).$ By our assumption $H^2(\mathfrak{g}, \mathfrak{g})=0$, one gets $f_r\in B^2(\mathfrak{g}, \mathfrak{g}),$ thus we can find $h_r\in C^1(\mathfrak{g}, \mathfrak{g})$ such that $f_r=\delta^{1}h_r.$ Put $\Phi_{t}=\mathrm{id}_{\mathfrak{g}}-h_rt^r,$ then $\Phi_{t}\circ(\mathrm{id}_{\mathfrak{g}}+h_rt^r+{h_{r}}^{2}t^{2r}+{h_{r}}^{3}t^{3r}+\cdots)
=(\mathrm{id}_{\mathfrak{g}}-h_rt^r)\circ(\mathrm{id}_{\mathfrak{g}}+h_rt^r+{h_{r}}^{2}t^{2r}+{h_{r}}^{3}t^{3r}+\cdots)=
(\mathrm{id}_{\mathfrak{g}}+h_rt^r+{h_{r}}^{2}t^{2r}+{h_{r}}^{3}t^{3r}+\cdots)-(h_rt^r+{h_{r}}^{2}t^{2r}+{h_{r}}^{3}t^{3r}+\cdots)=\mathrm{id}_{\mathfrak{g}},$
moreover, $(\mathrm{id}_{\mathfrak{g}}+h_rt^r+{h_{r}}^{2}t^{2r}+{h_{r}}^{3}t^{3r}+\cdots)\circ\Phi_{t}=\mathrm{id}_{\mathfrak{g}}.$  Hence $\Phi_{t}: \mathfrak{g}\rightarrow \mathfrak{g}$ is a linear isomorphism and $\Phi_{t}\circ\alpha=\alpha\circ\Phi_{t}.$ Set $f^{'}_{t}(x_{1},\cdots,x_{n})=\Phi_{t}^{-1}f_{t}(\Phi_{t}(x_{1}),\cdots,\Phi_{t}(x_{n})),$ then $f^{'}_{t}$ is also a  deformation of $(\mathfrak{g}, [\cdot,\cdots,\cdot],\alpha,\beta)$ and $f_{t}\sim f^{'}_{t}.$
Note that $\Phi_{t}f^{'}_{t}(x_{1},\cdots,x_{n})=f_{t}(\Phi_{t}(x_{1}),\cdots,\Phi_{t}(x_{n})).$ Let $f^{'}_{t}=\sum_{i\geq 0}f^{'}_{i}t^{i}.$ Then
$$(\mathrm{id}_{\mathfrak{g}}-h_rt^r)\sum_{i\geq 0}f^{'}_{i}(x_{1},\cdots,x_{n})t^{i}=(f_{0}+\sum_{i\geq r}f_{i}t^{i})(x_{1}-h_{r}(x_{1})t^{r},\cdots,x_{n}-h_{r}(x_{n})t^{r}).$$
So
\begin{align*}
&\sum_{i\geq 0}f^{'}_{i}(x_{1},\cdots,x_{n})t^{i}-\sum_{i\geq 0}h_{r}\circ f^{'}_{i}(x_{1},\cdots,x_{n})t^{i+r}\\
=&f_{0}(x_{1},\cdots,x_{n})-\sum_{i=1}^{n}f_{0}(x_{1},\cdots,h_{r}(x_{i}),\cdots,x_{n})t^{r}\\
+&\sum_{1\leq i<j\leq n}f_{0}(x_{1},\cdots,h_{r}(x_{i}),\cdots,h_{r}(x_{j}),\cdots,x_{n})t^{2r}\\
-&\sum_{1\leq i<j<k\leq n}f_{0}(x_{1},\cdots,h_{r}(x_{i}),\cdots,h_{r}(x_{j}),\cdots,h_{r}(x_{k}),\cdots,x_{n})t^{3r}+\cdots\\
+&(-1)^{n}f_{0}(h_{r}(x_{1}),h_{r}(x_{2}),\cdots,h_{r}(x_{n}))t^{nr}+\sum_{i\geq r}f_{i}(x_{1},\cdots,x_{n})t^{i}\\
-&\sum_{i\geq r}\sum_{j=1}^{n} f_{i}(x_{1},\cdots,h_{r}(x_{j}),\cdots,x_{n})t^{i+r}\\
+&\sum_{i\geq r}\sum_{1\leq j\leq k\leq n} f_{i}(x_{1},\cdots,h_{r}(x_{j}),\cdots,h_{r}(x_{k}),\cdots,x_{n})t^{i+2r}+\cdots.
\end{align*}
By the above equation, one gets
$$f^{'}_{0}(x_{1},\cdots,x_{n})=f_{0}(x_{1},\cdots,x_{n})=[x_{1},\cdots,x_{n}]_{\mathfrak{g}};$$
$$f^{'}_{1}(x_{1},\cdots,x_{n})=\cdots=f^{'}_{r-1}(x_{1},\cdots,x_{n})=0;$$
$$f^{'}_{r}(x_{1},\cdots,x_{n})-h_{r}[x_{1},\cdots,x_{n}]_{\mathfrak{g}}=-\sum_{i=1}^{n}[x_{1},\cdots,h_{r}(x_{i}),\cdots,x_{n}]_{\mathfrak{g}}+f_{r}(x_{1},\cdots,x_{n}).$$
Furthermore, we have $$f^{'}_{r}(x_{1},\cdots,x_{n})=-\delta^{1}h_{r}(x_{1},\cdots,x_{n})
+f_{r}(x_{1},\cdots,x_{n})=0,$$
hence, $f^{'}_{t}=f_{0}+\sum_{i\geq r+1}f^{'}_{i}t^{i}.$ By induction, one can prove $f_t\sim f_0,$ that is, $(\mathfrak{g}, [\cdot,\cdots,\cdot]_{\mathfrak{g}},\alpha,\beta)$ is analytically rigid.

\end{proof}

\medskip

\noindent{\bf Acknowledgements}
I would like to thank Nizar Ben Fraj, Sami Mabrouk and Liangyun Chen for
helpful discussions.

\label{lastpage-01}

\end{document}